\documentclass[11pt]{article}

\usepackage{mathrsfs}
\usepackage{amsfonts}
\usepackage{amsmath,extpfeil}
\usepackage{stmaryrd}
\usepackage{amssymb}
\usepackage{amsthm}
\usepackage{mathrsfs}
\usepackage{url}
\usepackage{amsfonts}
\usepackage{amscd}
\usepackage{indentfirst}
\usepackage{enumerate}
\usepackage{amsmath,amsfonts,amssymb,amsthm}
\usepackage{amsmath,amssymb,amsthm,amscd}
\usepackage{graphicx,mathrsfs}
\usepackage{appendix}
\usepackage[numbers,sort&compress]{natbib}
\usepackage{color}
\usepackage[colorlinks, linkcolor=blue, citecolor=blue]{hyperref}

\usepackage{sidecap}
\usepackage{float}
\usepackage{extarrows}
\usepackage{booktabs}
\usepackage{verbatim}

\usepackage[usenames,dvipsnames]{xcolor}

\newtheorem{thm}{Theorem}
\newtheorem{theorem}{Theorem}[section]
\newtheorem{lemma}[theorem]{Lemma}

\theoremstyle{definition}
\newtheorem{remark}[theorem]{Remark}

\def\XXint#1#2#3{{\setbox0=\hbox{$#1{#2#3}{\int}$}
         \vcenter{\hbox{$#2#3$}}\kern-.5\wd0}}

\def\R{\mathbb{R}}

\numberwithin{equation}{section}

\allowdisplaybreaks[4]
\begin{document}

\title{On mass - critical NLS with local and non - local nonlinearities
}

\author{Vladimir Georgiev \qquad Yuan Li\footnote{Corresponding author}}
\date{}
\maketitle

\begin{abstract}
We consider the following  nonlinear Schr\"{o}dinger equation with the double $L^2$-critical nonlinearities
\begin{align*}
iu_t+\Delta u+|u|^\frac{4}{3}u+\mu\left(|x|^{-2}*|u|^2\right)u=0\ \ \ \text{in $\mathbb{R}^3$,}
\end{align*}
where $\mu>0$ is small enough.  Our first goal is to prove the existence and the  non-degeneracy of the ground state $Q_{\mu}$. In particular, we develop an appropriate  perturbation approach to prove the radial non-degeneracy property and then obtain the general non-degeneracy of the ground state $Q_{\mu}$. We then show the existence of  finite time blowup solution with minimal mass $\|u_0\|_{L^2}=\|Q_{\mu}\|_{L^2}$. More precisely, we construct the minimal mass blowup solutions that are parametrized by the energy $E_{\mu}(u_0)>0$ and the momentum $P_{\mu}(u_0)$. In addition, the non-degeneracy property plays crucial role in this construction.\\
\noindent
{\bf keywords:} Nonlinear Schr\"odinger equation; Non-degeneracy; Mass-Critical; Blow-up; Local and  non - local nonlinearities
\end{abstract}

\section{Introduction}
\noindent
In this paper, we consider the following nonlinear Schr\"{o}dinger equation
\begin{equation}\label{equ:double}
\begin{cases}
i\partial_tu+\Delta u+|u|^{\frac{4}{3}}u
+\mu\left(|x|^{-2}*|u|^2\right)u=0,\,\,t\in\mathbb{R},\,x\in\mathbb{R}^3,\\
u(0,x)=u_0(x)\in H^1(\mathbb{R}^3),
\end{cases}
\end{equation}
where $u=u(t,x)$ is complex-valued function in time-space $\mathbb{R}\times\mathbb{R}^3$. 

{ It is well - known that the classical Schr\"odinger - Poisson - Slatter 
equation 
\begin{equation*}
\begin{cases}
i\partial_tu+\Delta u+|u|^{p-1}u
-\mu Au=0,\,\,t\in\mathbb{R},\,x\in\mathbb{R}^3,\\
-\Delta A = 4\pi |u|^2, ~~
u(0,x)=u_0(x)\in H^1(\mathbb{R}^3),
\end{cases}
\end{equation*}
is a model derived from  Poisson - Newton interaction \cite{SS2004JSP}. This equation can be considered as a generalization of \eqref{equ:double} with $\mu<0$ and it is intensively studied (see for example  \cite{R2010ARMA,Schlein:book,GPV2012Poincare,I2013TMNA} and references there). The equation \eqref{equ:double} can be rewritten as  
\begin{equation*}
\begin{cases}
i\partial_tu+\Delta u+|u|^{p-1}u
-\mu A u=0,\,\,t\in\mathbb{R},\,x\in\mathbb{R}^3,\\
(-\Delta)^{1/2} A = 4\pi |u|^2,~~
u(0,x)=u_0(x)\in H^1(\mathbb{R}^3),
\end{cases}
\end{equation*} 
can be considered as a modification of the Poisson equation for the gravitational potential that is typical in the study of fractional Newtonian gravity  as an alternative to standard Newtonian gravity ( \cite{I2013TMNA,G2020,GGV2020,V2020FP,V2021EP}).

Another, case of nonlinear Schr\"odinger  type equation with Hartree type nonlinearity is the following one
\begin{equation}\label{equH1}
i\partial_tu+\Delta u+Vu
-(w*|u|^2) u=0,\,\,t\in\mathbb{R},\,x\in\mathbb{R}^3.
\end{equation}

Among the several contexts of relevance of \eqref{equH1}, one is surely the quantum
dynamics of large Bose gases, where particles are subject to an external potential
$V$ and interact through a two-body potential $w$. In this case \eqref{equH1} emerges
as the effective evolution equation, rigorously  obtained through the  limit $N \to \infty$, where $N$ is the number of particles. The precise
meaning of the control of the many-body wave function is in the sense of one-body
reduced density matrices. 
 This model is also intensively studied for  sufficiently  large  class of cases,
ranging from bounded to locally singular potentials w, and through a multitude of
techniques to control the limit of infinitely many particles (see, e.g., \cite[Chapter 2]{Schlein:book} and the references therein). If we assume $w$ to be homogeneous function and the potential $V$ is of self interacting type, $V(u)=|u|^{p-1}$, then we arrive at the model \eqref{equ:double} with $\mu>0.$

}

Let us recall some basic facts about the Cauchy problem \eqref{equ:double}.
From  \cite{Cazenave:book}, it is known that \eqref{equ:double} is local well-posedness of \eqref{equ:double}. That is, given $u_0\in H^1(\mathbb{R}^3)$, there exists a unique maximal solution $u\in C\left((-T_{min},T_{max});H^1(\mathbb{R}^3)\right)$ to \eqref{equ:double} and there holds the blowup alternative:
\begin{align}\label{intro:blowup:alternative}
T<+\infty\,\,\,\text{implies}\,\,\,\lim_{t\rightarrow T}\|u(t)\|_{H^1}=+\infty.
\end{align}
Furthermore, the $H^1$ flow admits the conservation laws:
\begin{align*}
&\text{Mass:}~~M(u)(t)=\int|u(t,x)|^2=M(u_0);\\
&\text{Energy:}~E_{\mu}(u)(t)=\frac{1}{2}\int|\nabla u(x,t)|^2-\frac{3}{10}\int|u(x,t)|^{\frac{10}{3}}-\frac{\mu}{4}\int \frac{|u(t,x)|^2|u(t,y)|^2}{|x-y|^2}=E_{\mu}(u_0);\\
   &\text{Momentum:}~~ P_{\mu}(u(t))=\Im\int \bar{u}(t,x)\nabla{u}(t,x)dx=P_{\mu}(u_0).
\end{align*}
First, we recall the structure of the mass critical problem. In this case, the scaling symmetry
\begin{align*}
    u_a(t,x)=a^{\frac{3}{2}}u(a^2t,ax),\,\,\text{where}\,\,a>0,
\end{align*}
acts on the set of solutions and leaves the mass invariant
\begin{align*}
    \|u_a(t,\cdot)\|_{L^2}=\|u(a^2t,\cdot)\|_{L^2}.
\end{align*}

\subsection{The case \texorpdfstring{$\mu=0$}{mu=0}}
A criterion of global-in-time existence for $H^1$ initial data is derived by using the Gagliardo-Nirenberg inequality with the best constant
\begin{align*}
    \|u\|_{L^{\frac{10}{3}}}^{\frac{10}{3}}\leq C\|u\|_{L^2}^{\frac{4}{3}}\|\nabla u\|_{L^2}^2,
\end{align*}
where $C=\frac{5}{3}\frac{1}{\|Q\|_{L^2}^{\frac{4}{3}}}$, and $Q$   is the unique (\cite{BL1983ARMA,K1989ARMA}) up to symmetries solution to the  positive ground state equation
\begin{align*}
    -\Delta Q+Q-|Q|^{\frac{4}{3}}Q=0,\,\,Q(x)>0,\,\,Q\in H^1(\mathbb{R}^3).
\end{align*}
So that for all $u\in H^1(\mathbb{R}^3)$, we have
\begin{align}\notag
    E_0(u)\geq\frac{1}{2}\|\nabla u\|_{L^2}^2\left[1-\left(\frac{\|u\|_{L^2}}{\|Q\|_{L^2}}\right)^{\frac{4}{3}}\right],
\end{align}
which together with the conservation of mass, energy and the blowup criterion \eqref{intro:blowup:alternative} implies that the global existence of all solution with initial data $\|u_0\|_{L^2}<\|Q\|_{L^2}$.

At the mass critical level $\|u_0\|_{L^2}=\|Q\|_{L^2}$, the pseudo-conformal symmetry of \eqref{equ:double} yields an explicit minimal blowup solution:
\begin{align}\notag
    S(t,x)=\frac{1}{|t|^{\frac{3}{2}}}Q\left(\frac{x}{t}\right)e^{-i\frac{|x|^2}{4t}}e^{\frac{i}{t}},\,\,\|S(t)\|_{L^2}=\|Q\|_{L^2},\,\,\|\nabla S(t)\|_{L^2}\stackrel{t\rightarrow0^-}{\sim}\frac{1}{|t|}.
\end{align}
Merle \cite{M1993Duke} obtained the classification in the energy space of the minimal blowup elements; the only $H^1$ finite time blowup  solution with mass $\|u\|_{L^2}=\|Q\|_{L^2}$ is given by above up to the symmetries of the flow.

Note that the minimal blow up dynamic can be extended to the super-critical mass case $\|u_0\|_{L^2}>\|Q\|_{L^2}$ and that is corresponds to an unstable threshold dynamics between global in time scattering solutions and finite time blow up solutions in the stable blow up regime
\begin{align*}
    \|\nabla u(t)\|_{L^2}\sim\sqrt{\frac{\log|\log|T^*-t||}{T^*-t}},\,\,\text{as}\,\,t\sim T^*.
\end{align*}
For results about the existing literature for the $L^2$ critical blow up problem, one can see\cite{MR2004Invent,MR2005Ann,MRGFA2003,MR2005CMP,MR2006JAMS,MRS2013Amer}  and references therein.

\subsection{The case \texorpdfstring{$\mu<0$}{mu<0}}
Let us now consider the case of a defocusing perturbation. At the threshold, we claim:
\begin{lemma}\label{lemma1dou}
\textbf{(Global existence  for $\mu<0$).} Let $\mu<0$ and $u_0\in H^1(\mathbb{R}^3)$ with $\|u_0\|_{L^2}=\|Q\|_{L^2}$. Then the solution of \eqref{equ:double} is global and bounded in $H^1(\mathbb{R}^3)$.
\end{lemma}

The  proof follows from the standard concentration compactness argument, see Appendix \ref{sectionlemma}, The similar results for the double power nonlinear Schr\"odinger equation can be found in \cite{LMR2016RMI}. The global existence criterion of lemma \ref{lemma1dou} is sharp in the sense that for all $\delta>0$, we can build an $H^{1}(\mathbb{R}^3)$ finite time blw-up solution to \eqref{equ:double} with the initial data $\|u_0\|_{L^2}=\|Q\|_{L^2}+\delta$.

Now, we state the following blow up result.


\begin{lemma}\label{lemma12dou}
Let $\mu<0$ and close to $0$. For any $\delta>0$ there exists  $u_0\in H^1_{rad}(\mathbb{R}^3)$  such that $$xu_0 \in L^2(\mathbb{R}^3), \ \|u_0\|_{L^2}=\|Q\|_{L^2}+\delta$$ and the solution $u$ of \eqref{equ:double} blowup in finite time.
\end{lemma}

By the  Virial argument, we can obtain that the blowup solution with mass arbitrary close to (but larger than) the critical mass, see Appendix \ref{sectionlemma}.

\begin{remark}
 Similar questions can be addressed for the nonlocal  perturbation of the classical mass critical problem
 \begin{align}\notag
  iu_t+\Delta u+|u|^\frac{4}{3}u+\mu\left(|x|^{-\gamma}*|u|^p\right)|u|^{p-2}u=0\ \ \ \text{in $\mathbb{R}^3$,}
 \end{align}
 with $0<\gamma\leq2$ and  $2\leq p\leq 2+\frac{5-\gamma}{3}$.

 (i) If $\mu<0$ and initial data $u_0\in H^1(\mathbb{R}^3)$ with $\|u_0\|_{L^2}=\|Q\|_{L^2}$, then the solution is global.

 (ii) If $\mu<0$, $3p+\gamma\leq8$ and the initial data $u_0\in H^1_{rad}(\mathbb{R}^3)$ such that $xu_0\in L^2(\mathbb{R}^3)$, $\|u_0\|_{L^2}=\|Q\|_{L^2}+\delta$. Then the solution blowup in finite time.

 The analysis could also be extended to the higher dimensional case.
\end{remark}

\subsection{The case \texorpdfstring{$\mu>0$}{mu>0}}
In this section and what follows. For simplicity, we introduce the notation 
$$ A(u)(x) =  \left(|x|^{-2}*|u|\right). $$
We now turn to the case $\mu>0$ for the rest of paper, i.e., we consider the model
\begin{align}\label{equ1:double}
    i\partial_tu+\Delta u+|u|^{\frac{4}{3}}u+\mu A(u^2)u=0.
\end{align}

Now, we state our first main result. For the small $L^2$ solutions, there exist arbitrarily small solitary waves.

\begin{thm}\label{Theorem1}
\textbf{(Small solitary waves).} Let $\mu>0$ be small enough. There exists $\delta=\delta(\mu)>0$ such that  for all $a\in\left(0,\|Q\|_{L^2}^2-\delta(\mu)\right)$, where $Q$ is the unique radial positive ground state solution of equation
\begin{equation}\label{equation:mu0}
    -\Delta Q+Q=|Q|^{\frac{4}{3}}Q
\end{equation}
and the best constant  $C_*$ in  the Gagliardo-Nirenberg's inequality
\begin{align}\label{GN:nonlocaldou}
	\|A(|u|^2)|u|^2\|_{L^1} \leq C_* \|\nabla u\|_{L^2}^2 \|u\|_{L^2}^2.
\end{align}
Then there exists a positive Schwartz radially symmetric solution $Q_{\mu}$ of
\begin{align*}
    \Delta Q_{\mu}- Q_{\mu}+Q_{\mu}^{\frac{7}{3}}+\mu A(Q_{\mu}^2)Q_{\mu}=0,\,\,\|Q_{\mu}\|_{L^2}^2=a.
\end{align*}
In addition, define the linear operator $L_{+,\mu}$ and $L_{-,\mu}$ associated to $Q_{\mu}$ by
\begin{align}\label{linear:operatordouble}
L_{+,\mu}\xi=&-\Delta\xi+\xi-\frac{7}{3}Q_{\mu}^{\frac{4}{3}}\xi- 2\mu A(Q_{\mu}\cdot\xi)Q_{\mu}-\mu A\left(Q_{\mu}^2\right)\xi,\notag\\
L_{-,\mu}\xi=&-\Delta\xi+\xi-Q_{\mu}^{\frac{4}{3}}\xi-\mu A\left(Q_{\mu}^2\right)\xi,
\end{align}
acting on $L^2(\mathbb{R}^3)$ with form domain $H^1(\mathbb{R}^3)$, where $\xi\in H^1(\mathbb{R}^3)$. We have the following  non-degeneracy result.
\begin{align*}
  \ker L_{+,\mu}=\{\nabla Q_{\mu}\},~~
  \ker L_{-,\mu}=\{Q_{\mu}\}.
\end{align*}
\end{thm}

\textbf{Comments on Theorem \ref{Theorem1}:}

1. Existence. From the standard variational argument, we can easily obtain the existence.

2. The kernel of $L_{-,\mu}$. By using the Sturm argument, we can obtain the  $\ker L_{-,\mu}=\{Q_{\mu}\}$. Here we do not need to assume that the parameter $\mu$  is small enough.

3. The kernel of $L_{+,\mu}$. This case seems more difficult, First, we restrict our attention to the case of radial Sobolev space $H^1_{rad}(\mathbb{R}^3)$. Here we develop an appropriate perturbation approach, together with the kernel of linear operators $L_{+,0}$ and $L_{-,0}$ to prove this result.  On the other hand, we can easily obtain that $Q_{\mu}\rightarrow Q_0$ in $H^1_{rad}(\mathbb{R}^3)$ as $\mu \to 0$, but this is not enough. We need a more precise estimate on the rate of convergence, namely we prove
\begin{align*}
    \|Q_{\mu}-Q\|_{H^2}\lesssim \mu.
\end{align*}
Here we assume that $\mu$ is small enough, the estimate is more delicate problem for $\mu$ large.

\begin{remark}
 So far, only a few articles have considered the uniqueness and non-degeneracy of the non-local nonlinear Schr\"{o}dinger equation, see \cite{KLR2009poincare,L2009APDE,X2016CVPDE}. The uniqueness problem without nondegeneracy  is treated in \cite{GS2018PD,GTV2019NA,L1976SAM,MZ2010ARMA}. For the general non-local nonlinear Schr\"{o}dinger equation or Choquard equation, the non-degeneracy property is still an open.
\end{remark}

A second main result is the existence of a minimal mass blowup solution for \eqref{equ1:double}.
\begin{thm}\label{theorem:minimialD}
\textbf{(Existence of minimal mass blowup elements).} Let $u_0\in H^1(\mathbb{R}^3)$ and $\mu>0$ be small enough. For $E_{\mu}(u_0)\in\mathbb{R}^*_+$, $P_{\mu}\in\mathbb{R}^3$, there exist $t^*<0$ and a  minimal mass solution $u\in\mathcal{C}\left([t^*,0);H^1(\mathbb{R}^3)\right)$ of equation \eqref{equ1:double} with
\begin{align*}
    \|u\|_{L^2}=\|Q_{\mu}\|_{L^2},\,\,E_{\mu}(u)=E_{\mu}(u_0),\,\,P_{\mu}(u)=P_{\mu}(u_0),
\end{align*}
which blows up at time $T=0$. More precisely, it holds that
\begin{align*}
    u(t,x)-\frac{1}{\lambda^{\frac{3}{2}}(t)}Q_{\mu}\left(\frac{x-\alpha(t)}{\lambda(t)}\right)e^{i\gamma(t)}\rightarrow0\,\,\text{in}\,\,L^2(\mathbb{R}^3)\,\,\text{as}\,\,t\rightarrow0^-,
\end{align*}
where
\begin{align*}
    \lambda(t)=\lambda^*t+\mathcal{O}(t^3),\,\,\,\gamma(t)=\frac{1}{\lambda^*|t|}+\mathcal{O}(t^2),\,\,\alpha(t)=x_0+\mathcal{O}(t^3),
\end{align*}
with some constant $\lambda^*>0$, and the blowup speed is given by
\begin{align*}
    \|\nabla u(t)\|_{L^2}\sim\frac{C(u_0)}{|t|},\,\,\text{as}\,\,t\rightarrow0^-,
\end{align*}
where $C(u_0)$ is a constant only depend on the initial data $u_0$.
\end{thm}

\textbf{Comments on the result.}

1. $\mu>0$ is small. In the present work, we assume that $\mu>0$ is small enough. This condition guarantee the existence of $Q_{\mu}$ and the radial non-degeneracy of the linearized operator  $L_{+,\mu}$. On the other hand, $\mu>0$ is small plays an important role in the refine energy estimate.

2. On the minimal elements. 

For an inhomogeneous problem
\begin{align*}
    i\partial_tu+\Delta u-V(x)u+k(x)|u|^\frac{4}{N}u=0.
\end{align*}
When $V(x)=0$, $k(x)\neq0$ and $N=2$, Rapha\"{e}l and Szeftel \cite{RS2011JAMS} obtained the existence and uniqueness of the minimal mass blowup solution under a necessary and sufficient condition on $k(x)$, in the absence of pseudo-conformal transformation. When $V(x)\neq0$, $k(x)\neq0$ and $N=1,2$,  Banica,  Carles and Duyckaerts  \cite{BCD2011CPDE} proved the existence of the minimal mass blowup solution. On the other hand, Le Coz, Martel and   Rapha\"{e}l \cite{LMR2016RMI} also considered the double power nonlinear Schr\"{o}dinger equation
\begin{align}\notag
    i\partial_tu+\Delta u+|u|^{4/d}u+|u|^{p-1}u=0,\,\,1<p<1+\frac{4}{d},\,\,d\leq3,
\end{align}
and obtained the existence of finite time blow up minimal solutions in the radial case.

For the mass critical nonlocal problem such as half wave equation
\begin{align}\label{equ:hw}
    i\partial_tu+\sqrt{-\Delta}u+|u|^{\frac{2}{N}}=0,
\end{align}
(the energy space for \eqref{equ:hw} is $H^{\frac{1}{2}}(\mathbb{R}^N)$), when $N=1$, an existence result similar to Theorem \ref{theorem:minimialD} was proved by  Krieger, Lenzmann and Rapha\"{e}l \cite{KLR2013ARMA}; when $N=2,3$, the existence result in the Sobolev space was proved by the authors in the present paper \cite{GL2021CPDE,GL2021JFA}. Also, Lan \cite{Lan2021IMRN} obtained the blowup solution for the general fractional Schr\"odinger equation in the one dimension case. For the other constructions of minimal mass solutions for dispersive equations, such as Martel and Pilod \cite{MP2017MA} addressed the case of the modified Benjamin-Ono equation, which also involves the nonlocal operator $D$.

This paper is organized as follows: in Section 2, we prove the Theorem \ref{Theorem1}; in section 3, we use the result of Theorem \ref{Theorem1} to construct the approximate blowup profile $R_{\mathcal{P}}$; In section 4, we establish the energy, modulation estimates and the refined energy/Morawetz type estimate, which will be a key ingredient in the compactness argument to construct minimal mass blowup solutions; In section 5, we prove the main result Theorem \ref{theorem:minimialD}; and the finally section is the Appendix.

\textbf{Notations and definitions}\\
- $(f,g)=\int \bar{f}g$ as the inner product on $L^2(\mathbb{R}^3)$.\\
- $\|\cdot\|_{L^p}$ denotes the $L^p(\mathbb{R}^3)$ norm for $p\geq 1$.\\
- $\widehat{f}$ denotes the Fourier transform of function $f$.\\
- We shall use $X\lesssim Y$ to denote that $X\leq CY$ holds, where the constant $C>0$ may change from line to line, but $C$ is allowed to depend on universally fixed quantities only.\\
- Likewise, we use $X\sim Y$ to denote that both $X\lesssim Y$ and $Y\lesssim X$ hold.\\
- $\Re{f}$ and $\Im{f}$ denote the real part and imaginary part of function $f$, respectively.

\section{Existence and  Non-degeneracy}
In this section, we consider the existence and  non-degeneracy of the ground state solution  of \eqref{equ1:double}. Now, we introduce the minimization problem
\begin{align}\label{min:double3D}
e_{\mu}=\inf_{\|u\|_{L^2}^2=a}\{E_{\mu}(u):u\in H^1(\mathbb{R}^3)\}.
\end{align}
This lemma will give the existence and properties of the minimizer.

\begin{lemma}
Let $\mu>0$ be small enough. There exists $\delta=\delta(\mu)>0$ such that  the constrained minimization problem \eqref{min:double3D}    with $ a <\|Q\|_{L^2}^2-\delta(\mu)$, where $C_*$ is the best constant of the Gagliardo-Nirenberg's inequality \eqref{GN:nonlocaldou}, has a minimizer  $\phi_{\mu}\in H^1(\mathbb{R}^3)$, that after rescaling  satisfies
\eqref{equ1:double} and it is radial and symmetry decreasing function.
\end{lemma}

\begin{proof}
First, we show that $E_{\mu}(u)$ is bounded from below, when $u$ obeys the constraint $\|u\|_{L^2}^2=a <\|Q\|_{L^2}^2-\delta(\mu)$. Indeed, by the Gagliardo-Nirenberg's inequalities and Hardy-Littlewood-Sobolev inequality, we have
\begin{align}\notag
E_{\mu}(u)\geq&\frac{1}{2}\|\nabla u\|_{L^2}^2-\frac{1}{2}\frac{\|u\|_{L^2}^{\frac{4}{3}}}{\|Q\|_{L^2}^{\frac{4}{3}}}\|\nabla u\|_{L^2}^2-\mu C_*\|\nabla u\|_{L^2}^2\|u\|_{L^2}^{2}
=\frac{1}{2}\left(1-\frac{\|u\|_{L^2}^{\frac{4}{3}}}{\|Q\|_{L^2}^{\frac{4}{3}}}-2\mu C_*a\right)\|\nabla u\|_{L^2}^2.
\end{align}
From the assumption of $a$, we deduce that $E_{\mu}(u) \geq 0$.

Next, we discuss the existence and the properties of
the constrained minimizers. Taking a minimizing sequence $\{u_n\}$ and $\lim\limits_{n\rightarrow\infty}E_{\mu}(u_n)=e_{\mu}$. By the Riesz rearrangement inequality, we have
\begin{align*}
&\|\nabla u_n\|_{L^2}\geq\|\nabla u_n^*\|_{L^2},~~\|u_n\|_{L^q}=\|u_n^*\|_{L^q},\,\,\text{where}\,\,q\in[2,6]\\
&\int A(u_n^2)(x)u_n^2(x)dx\leq
\int A((u_n^*)^2)(x)(u_n^*)^2(x)dx.
\end{align*}
Combining the above relations, we have $E(u_n)\geq E(u_n^*)$ while $\|u_n^*\|_{L^2}^2=a$. Hence, $\mathop{\lim}\limits_ {n\rightarrow\infty}E_{\mu}(u_n^*) =e_{\mu}$ and $u_n^*$ is uniformly bounded sequence in $H^1(\mathbb{R}^3)$.

Moreover, $u_n^*$ are radial symmetry functions in the unit sphere of $L^2(\mathbb{R}^3)$. So we have a weakly convergent subsequence converging weakly in $L^2(\mathbb{R}^3)$. By the lower semi-continuity of the norm $\|\phi_{\mu}\|_{L^2}\leq a$ and
\begin{align}\notag
\liminf_{n\rightarrow\infty} \|\nabla u_n^*\|_{L^2}\geq\|\nabla \phi_{\mu}\|_{L^2}.
\end{align}
We also have that for every $|x|\geq0$,
\begin{align}\notag
a=\int|u_n^*|^2\geq\int_{|y|\leq|x|}|u_n^*(y)|^2dy\geq C |\cdot|^3|u_n^*(x)|^2,
\end{align}
whence $|u_n^*(x)|\leq C|x|^{-3/2}$ for every $x\in\mathbb{R}^3$.

It follows that $\{u_n^*\}$ is a compact sequence $L^q(\mathbb{R}^3)$, $2\leq q<6$ (Rellich-Kondrachov's). Hence, we can assume (after taking subsequence), $\lim\limits_{n\rightarrow\infty}\|u_n^*-\phi\|_{L^q}=0$ for any $2\leq q<6$. As a consequence, by the triangle inequality and the Hardy-Littlewood-Sobolev inequality or see \cite[Lemma 2.1]{LZW2019ZAMP}, we can obtain
\begin{align}\notag
\lim_{n\rightarrow\infty}\int A(|u_n^*|^2)|u_n^*(x)|^2dx=\int A(\phi^2)\phi(x)^2dx.
\end{align}
Next, we will show that $\{u_n^*\}$ converges to $\phi$ in $H^1(\mathbb{R}^3)$. In fact, we have
\begin{align*}
    \|\nabla u_n^*\|_{L^2}^2-\|\nabla\phi\|_{L^2}^2=&2E_{\mu}(u_n^*)-2E_{\mu}(\phi)+\frac{20}{3}\int|u_n^*|^{\frac{10}{3}}-\frac{20}{3}\int|\phi|^{\frac{10}{3}}\\
    &+\frac{\mu}{2}\int A(|u_n^*|^2)|u_n^*(x)|^2dx-\frac{\mu}{2}\int A(\phi^2)\phi(x)^2dx\\
    \leq&2E_{\mu}(u_n^*)-2e_{\mu}(\lambda)+\frac{20}{3}\int|u_n^*|^{\frac{10}{3}}-\frac{20}{3}\int|\phi|^{\frac{10}{3}}\\
    &+\frac{\mu}{2}\int A(|u_n^*|^2)|u_n^*(x)|^2dx-\frac{\mu}{2}\int A(\phi^2)\phi(x)^2dx\rightarrow0,
\end{align*}
as $n\rightarrow\infty$. From this we have $\limsup\limits_{n\to\infty}\|\nabla u_n^*\|_{L^2}^2\leq\|\nabla\phi\|_{L^2}^2$. Combining the weakly lower semi-continuity, we can obtain that
$\|u_n^*-\phi\|_{H^1}\rightarrow0$. And now this lemma is proved.
\end{proof}

Since $\phi_{\mu}$ is a minimizer of \eqref{min:double3D}, it satisfies the Euler-Lagrange equation
\begin{equation*}
    -\Delta\phi_{\mu}+\beta_{\mu}\phi_{\mu}=|\phi_\mu|^{\frac{4}{3}}\phi_\mu
    +\mu A(\phi_{\mu}^2)|\phi_{\mu}.
\end{equation*}
Multiplying both side by $\phi_{\mu}$ and then integrate by part, we obtain
\begin{align*}
    \beta_{\mu}\int|\phi_{\mu}|^2=-\int|\nabla\phi_{\mu}|^2+\int|\phi_{\mu}|^{\frac{10}{3}}+\mu\int A(\phi_{\mu}^2)|\phi_{\mu}|^2.
\end{align*}
On the other hand,  we have the Pohozaev identity
\begin{align}\label{Pohozaev:identitydou}
    \frac{1}{2}\int|\nabla\phi_{\mu}|^2+\frac{3}{2}\int|\phi_{\mu}|^2=\frac{9}{10}\int|\phi_{\mu}|^{\frac{10}{3}}+\mu\int A(\phi_{\mu}^2)|\phi_{\mu}|^2.
\end{align}
Combining the above two identities, we can obtain
\begin{align*}
    \beta_{\mu}\int|\phi_{\mu}|^2=\frac{2}{5}\int|\phi_{\mu}|^{\frac{10}{3}}+\mu\frac{
    1}{2}\int A(\phi_{\mu}^2)|\phi_{\mu}|^2>0.
\end{align*}
This implies that $\beta_{\mu}>0$. Let $\phi_{\mu}(x)=\beta_{\mu}^{\frac{3}{4}}Q_{\mu}\left(\sqrt{\beta_{\mu}}x\right)$, then $\|\phi_{\mu}\|_{L^2}^2=\|Q_{\mu}\|_{L^2}^2$ and
\begin{equation}\label{equ:ell1}
    -\Delta Q_{\mu}+Q_{\mu}=|Q_{\mu}|^{\frac{4}{3}}Q_{\mu}+\mu A(Q_{\mu}^2)Q_{\mu}.
\end{equation}

Throughout this paper, we denote the linearized operator (with respect to complex-valued functions) close to the ground state $Q_{\mu}$ by
\begin{align}\notag
L_{\mu}=\left[\begin{array}{cc}L_{+,\mu}&0\\0&L_{-,\mu}\end{array}\right],
\end{align}
with the scalar self-adjoint operators $L_{+,\mu},\,L_{-,\mu}$ defined in \eqref{linear:operatordouble}.

For radial $\xi\in L^2(\mathbb{R}^3)$ in the kernel of $L_{-,\mu}$ we can write
\begin{equation}\notag
-  \xi^{\prime\prime}(r) - \frac{2}{r} \xi^\prime(r) + \xi(r) - Q_{\mu}^{\frac{4}{3}}\xi -\mu A(Q_{\mu}^2)\xi =0.
\end{equation}
Using the  argument from \cite[Lemma 1]{GS2018PD} and assuming $\xi \perp Q_{\mu}$ we easily obtain $\xi =0.$
Also we have
\begin{align*}
    ( L_{-,\mu} u,u )_{L^2} \geq \delta \|u\|_{L^2}^2~~\text{for}~~u \perp Q_{\mu}.
\end{align*}

Indeed, we first prove that the operator $L_{-,\mu}$ is non-negative. Assume that $L_{-,\mu}$ has a negative eigenvalue, say $-\sigma^2.$ Without loss of generality, we may assume that it is the smallest eigenvalue, so that,
\begin{align}\label{sec2min}
    -\sigma^2=\inf_{\|\psi\|_{L^2}^2=1}(L_{-,\mu}\psi,\psi).
\end{align}
The corresponding eigenfunction, say $\phi$ can be constructed as a minimizer of the minimization problem \eqref{sec2min}. It standard implication that the symmetric-decreasing rearrangement $\phi^*$ of $\phi$ is also minimizer, since we have
\begin{align*}
\|\nabla \psi\|_{L^2}\geq\|\nabla
\psi^*\|_{L^2},~
\int A(\psi^2)|\psi(x)|^2dx\leq
\int A((\psi^*)^2)|\psi^*(x)|^2dx,~
\|\psi\|_{L^q}=\|\psi^*\|_{L^q}.
\end{align*}
Thus $(L_{-,\mu}\psi,\psi)\geq(L_{-,\mu}\psi^*,\psi^*)$. It follows that  $\phi^*$ is also minimizer and   $\phi^* \geq 0$. But if such eigenfunction corresponds to a negative eigenvalue, then it must be perpendicular to the eigenfunction $Q_{\mu}$ corresponds to eigenvalue zero. However, both $\phi>0$ and $Q_{\mu}>0$, so we have  a contradiction. It follows that $L_{-,\mu}\geq 0$.

Next, we prove that $\ker L_{-,\mu}$ is one dimensional space generated by $Q_\mu.$ Indeed, if $\xi \in \ker L_{-,\mu}$   and $\xi\perp Q_{\mu}$, we
have the equation
$$-\Delta\xi(r)+\xi(r)-V(r) \xi(r) = 0, \  V(r) =Q_{\mu}^{\frac{4}{3}}+\mu A\left(Q_{\mu}^2\right).$$
Since $Q_\mu$ satisfies the same equation, the Sturm argument shows that between any two zeros of $\xi$ there is a zero of $Q_\mu$ and this is a contradiction.

\subsection{Limit of~\texorpdfstring{$Q_\mu$}{Q-mu} and radial non-degeneracy of \texorpdfstring{$L_{+,\mu}$}{Lmu}}
To show that the kernel of $L_{+,\mu}$ in $H^1_{rad}(\mathbb{R}^3)$ is trivial we have to show that the system
\begin{equation}\notag
-  \xi^{\prime\prime}(r) - \frac{2}{r} \xi^\prime(r)+\xi - \frac{7}{3}Q_{\mu}^{\frac{4}{3}}\xi -\mu A(Q_{\mu}^2)\xi - 2\mu A(Q_{\mu}\cdot\xi))Q_{\mu}=0.
\end{equation}
has only trivial solutions.

To prove the triviality of $  \ker L_{+,\mu},$ we shall make appropriate expansions of $Q_\mu$ and $  L_{+,\mu}$ around $\mu =0.$

We know that $Q_\mu$ is a solution to
\begin{equation}\notag
-  Q_\mu^{\prime\prime}(r) - \frac{2}{r} Q_\mu^\prime(r)+Q_{\mu} - Q_\mu^{\frac{7}{3}} -\mu A(Q_{\mu}^2)Q_\mu =0.
\end{equation}
Next lemma we will give the relations between $Q_{\mu}$ and $Q$, where $Q$ is the ground state of equation \eqref{equation:mu0}.
\begin{lemma}
One can show that
\begin{align*}
   Q_\mu \rightarrow Q_0=Q,~~\text{as $\mu \to 0$ in $H^1(\mathbb{R}^3)$}. 
\end{align*}
\end{lemma}
\begin{proof}
Since $\|Q_{\mu}\|_{L^2}\leq \|Q_0\|_{L^2} $ is uniformly bounded, we only have to derive a uniformly bounded for $\|\nabla Q_{\mu}\|_{L^2}$, which can be done as follows. Note that $Q_{\mu}$ satisfies the equation
\begin{align}\notag
-\Delta Q_{\mu}+Q_{\mu}=|Q_{\mu}|^{4/3}Q_{\mu}+\mu A(Q_{\mu}^2)Q_{\mu}.
\end{align}
Multiplying both sides by $Q_{\mu}$ and then integrate by part 
\begin{align}\label{eq:ie1}
    \int|\nabla Q_{\mu}|^2+\int|Q_{\mu}|^2=\int|Q_{\mu}|^{\frac{10}{3}}+\mu \int A(Q_{\mu}^2)Q_{\mu}^2
\end{align}
and the Pohozaev identity
 we can obtain that
\begin{align}\notag
    \int|Q_{\mu}|^2=\frac{2}{5}\int|Q_{\mu}|^{\frac{10}{3}}+\frac{\mu}{2}\int A(Q_{\mu}^2)Q_{\mu}^2> \frac{2}{5}\int|Q_{\mu}|^{\frac{10}{3}}.
\end{align}
Hence, we have
\begin{align}\notag
     \int |\nabla Q_{\mu}|^2+ |Q_{\mu}|^2<& \frac{5}{2}\int|Q_{\mu}|^2+\mu\int A( Q_{\mu}^2) Q_{\mu}^2
     \leq \frac{5}{2}\int|Q_{\mu}|^2+\mu C\|Q_{\mu_n}\|_{L^2}^2\int|\nabla Q_{\mu}|^2.
\end{align}
Hence, $\|\nabla Q_{\mu_n}\|_{L^2}$ is uniformly bounded for $\mu$ sufficiently small.

The above argument implies that $\{Q_{\mu_n}\}$ is uniformly bounded in $H^1_{rad}(\mathbb{R}^3)$. Therefore, we can assume that, up to a subsequence, still denote $Q_{\mu_n}$, such that $Q_{\mu_n}$ converges weakly to a non-negative radial function $Q_0\in H^1_{rad}(\mathbb{R}^3)$, that is
\begin{align}\notag
    Q_{\mu_n}\rightharpoonup Q_0\,\,\text{weakly in}\,\,H^1_{rad}(\mathbb{R}^3).
\end{align}
Moreover, by the compact embedding $H^1_{rad}(\mathbb{R}^3)\hookrightarrow L^q(\mathbb{R}^3)$ for any $2<q<6$ (see Strauss \cite{S1977CMP}), we can assume that
\begin{align*}
    &Q_{\mu_n}\rightarrow Q_0\,\,\text{in}\,\,L^q(\mathbb{R}^3)~~\text{for any $2<q<6$}\\
    &Q_{\mu_n}\rightarrow Q_0,\,\,\,a.e.\,\text{in}\,\,\mathbb{R}^3.
\end{align*}
From the Hardy-Littlewood-Sobolev inequality (see \cite{LZW2019ZAMP,L2001:book}), we easily deduce that
\begin{align*}
    \lim_{n\rightarrow+\infty}\int A(Q_{\mu_n}^2)Q_{\mu_n}^2=\int A(Q_{0}^2)Q_{0}^2.
\end{align*}
Furthermore, from \eqref{eq:ie1} and the above considerations  we obtain
\begin{align*}
    \|Q_{\mu_n}\|_{H^1}\rightarrow\|Q_0\|_{H^1}.
\end{align*}
Combining this with the weak convergence of $Q_{\mu_n}$, we obtain the strong convergence of $Q_{\mu_n}\rightarrow Q_0$ in $H^1(\mathbb{R}^3)$.
 In particular from  $Q_{\mu_n}\rightarrow Q_0$ in $L^2(\mathbb{R}^3)$ we see that $\|Q_0\|_{L^2}^2=\|Q\|_{L^2}^2.$
Finally, we will show that $Q_0=Q$ is the ground state solution of equation
\begin{equation}\label{equ:mu0}
    -\Delta u+u=|u|^{\frac{4}{3}}u.
\end{equation}
Indeed, from above argument, $Q_0$ is a (weak) solution of equation \eqref{equ:mu0}.
In fact, the identity $-\Delta Q_0 + Q_0-|Q_0|^{\frac{4}{3}}Q_0=0$ is fulfilled in $H^{-1}$ sense, since $Q_0 \in H^1.$ From this fact we obtain that $Q_0 \in H^2.$ And by the weakly lower semicontinuous
$$\|Q\|_{L^2}^2 = \|Q_0\|_{L^2}^2\leq\liminf_{\mu_n\rightarrow0}\|Q_{\mu_n}\|_{L^2}^2\leq\|Q\|_{L^2}^2.$$
This implies that $Q_0$ is a radial positive $L^2$ normalized solution of equation \eqref{equ:mu0} and using the uniqueness of the solution of equation \eqref{equ:mu0} we deduce $Q_0=Q.$
\end{proof}

Next lemma we will give the regularity an estimate concerning the linearized operators $L_{+,0}$ and $L_{-,0}$.
\begin{lemma}\label{lemma:operatorinverse}
Let $f,g\in H^1(\mathbb{R}^3)$ and suppose $g\perp Q$ and $f\perp\nabla Q$. Then we have the regularity bound
\begin{align}\notag
    \|L_{-,0}^{-1}g\|_{H^2}\lesssim \|g\|_{L^2}\,\,\text{and}\,\, \|L_{+,0}^{-1}f\|_{H^2}\lesssim \|f\|_{L^2}.
\end{align}
In particular, if $f,g\in H^1_{rad}(\mathbb{R}^3)$ and $g\perp Q$, we also have the same estimate.
\end{lemma}
\begin{proof}
From \cite{CGN2007SIAM,W1985SIAM}, it is known that
\begin{align*}
    \ker L_{+,0}=span\{\nabla Q\},\,\,\ker L_{-,0}=span\{Q\}.
\end{align*}
By the standard argument, we can easily obtain this lemma, here we omit the details.
\end{proof}

Since $Q_{\mu}$ satisfies
\begin{align}\notag
    L_{-,\mu}Q_{\mu} - \left( L_{-,0}Q_{\mu}+\mu A(Q_{\mu}^2)Q_{\mu}\right) = Q_\mu^{7/3}-Q_0^{4/3}Q_\mu-\mu A(Q_{\mu}^2)Q_{\mu} = \mathcal{O}(\mu).
\end{align}
Let $Q_{\mu}=Q+\mu\xi_{\mu}$, where $\xi_{\mu}\perp Q$. Then
\begin{align}\notag
    L_{+,0}\xi_{\mu}+A(Q_{\mu}^2)Q_{\mu}+\mathcal{O}(\mu)=0.
\end{align}
From this and the above lemma \ref{lemma:operatorinverse}, we have
\begin{align*}
    \|\xi_{\mu}\|_{H^2}\lesssim \|L_{+,0}^{-1}\left(A(Q_{\mu}^2)Q_{\mu}+\mathcal{O}(\mu)\right)\|_{H^2}\lesssim\|A(Q_{\mu}^2)Q_{\mu}+\mathcal{O}(\mu)\|_{L^2}
    \lesssim \|A(Q_{\mu}^2)Q_{\mu}\|_{L^2}+\mathcal{O}(\mu).
\end{align*}
By the H\"{o}lder inequality and Young inequality, we have
\begin{align}\notag
    \|A(Q_{\mu}^2)Q_{\mu}\|_{L^2}\leq \|A(Q_{\mu}^2)\|_{L^6}\|Q_{\mu}\|_{L^3}\leq \|Q_{\mu}\|_{L^4}^2\|Q_{\mu}\|_{L^3}.
\end{align}
Therefore, we deduce that
\begin{align}\notag
    \|Q_{\mu}-Q\|_{H^2}=\mu\|\xi_{\mu}\|_{H^2}\leq C\mu.
\end{align}

As an important result, we prove the so-called non-degeneracy of $L_{+,\mu}$ on $L^2_{rad}(\mathbb{R}^3)$;  that is, the triviality of its kernel.
\begin{lemma}\label{lemma:raidalnondegeneracy}
The linear operator $L_{+,\mu}$ given by \eqref{linear:operatordouble} satisfies the estimate
\begin{align*}
    \ker L_{+,\mu}=\{0\},
\end{align*}
when $L_{+,\mu}$ is restricted to $L^2_{rad}(\mathbb{R}^3)$ and $\mu$ is sufficiently small.
\end{lemma}
\begin{proof}
Assume that $\xi_{\mu}\in L^2_{rad}(\mathbb{R}^3)$ and $\xi_{\mu}\perp Q_{\mu}$ satisfies
\begin{align*}
    L_{+,\mu}\xi_{\mu}=0.
\end{align*}
Without loss of generality, we assume that $\|\xi_{\mu_k}\|_{L^2}=1$ and
\begin{align*}
    \xi_{\mu_k}=\alpha_{\mu_k}Q+\eta_{\mu_k},
\end{align*}
where $\eta_{\mu_k}\in H^1_{rad}(\mathbb{R}^3)$ and $\eta_{\mu_k}\perp Q$.

Now, we claim that $\alpha_{\mu_k}\rightarrow0$ and $\|\eta_{\mu_k}\|_{L^2}\rightarrow1$ as $\mu_k\rightarrow0$.
Indeed, if it is not true, then there exist $\alpha_0\neq 0$ and $\eta_0 \in H^1_{rad} (\mathbb{R}^3)$  so that (choosing suitable subsequences) we have $\alpha_{\mu_k}\rightarrow\alpha_0\neq0,$ $\eta_{\mu_k}\rightharpoonup \eta_0\neq 0,$ $\eta_0\perp Q$ so we have
\begin{align*}
    \xi_{\mu_k}=\alpha_{\mu_k}Q+\eta_{\mu_k}\rightharpoonup\alpha_0 Q+\eta_0 \neq 0.
\end{align*}
From this, we deduce that
\begin{align*}
    0=L_{+,\mu}\xi_{\mu_k}\rightharpoonup L_{+,0}(\alpha_0Q+\eta_0)=0
\end{align*}
and we see that $\alpha_0 Q+\eta_0 \in H^2_{rad} (\mathbb{R}^3).$
On the other hand, the kernel of $L_{+,0}$ in $H^2_{rad} (\mathbb{R}^3)$ is trivial, so we arrive at a contradiction and the claim is true.

Now, we can rewrite the equation $ L_{+,\mu}\xi_{\mu}=0$ as
\begin{align*}
    \alpha_{\mu_k}\left(-\Delta+1-\frac{7}{3}Q_{\mu_k}^{\frac{4}{3}}\right)Q-\alpha_{\mu_k}\mu_k\mathcal{K}(Q)+\left(-\Delta+1-\frac{7}{3}Q_{\mu_k}^{\frac{4}{3}}\right)\eta_{\mu_k}-\mu_k\mathcal{K}(\eta_{\mu_k})=0,
\end{align*}
where
\begin{align*}
    \mathcal{K}(\xi)= A(Q_{\mu_k}^2)\xi+2 A(Q_{\mu_k}\cdot\xi)Q_{\mu_k}.
\end{align*}
That is
\begin{align*}
    L_{+,0}\eta_{\mu_k}+\alpha_{\mu_k}L_{+,0}Q-\alpha_{\mu_k}\mu_k\mathcal{K}(Q)-\mu_k\mathcal{K}(\eta_{\mu_k})+\frac{7}{3}\left(Q^{\frac{4}{3}}-Q_{\mu_k}^{\frac{4}{3}}\right)(\alpha_{\mu_k}Q+\eta_{\mu_k})=0.
\end{align*}
Since $Q$ and $\eta_{\mu_k}$ are the radial functions,  by the lemma \ref{lemma:operatorinverse}, we have
\begin{align*}
\|\eta_{\mu_k}\|_{H^2}\leq&\alpha_{\mu_k}\|Q\|_{H^2}+\alpha_{\mu_k}\mu_k\|\mathcal{K}(Q)\|_{L^2}+\mu_k\|\mathcal{K}(\eta_{\mu_k})\|_{L^2}+\left\|\left(Q^{\frac{4}{3}}-Q_{\mu_k}^{\frac{4}{3}}\right)(\alpha_{\mu_k}Q+\eta_{\mu_k})\right\|_{L^2}\\
\leq&\alpha_{\mu_k}\|Q\|_{H^2}+\alpha_{\mu_k}\mu_k C+\mu_k C\|\eta_{\mu_k}\|_{H^2}+\left\|\left(Q^{\frac{4}{3}}-Q_{\mu_k}^{\frac{4}{3}}\right)\eta_{\mu_k}\right\|_{L^2}+C\mu_k^{\frac{4}{3}}\alpha_{\mu_k}.
\end{align*}
Let $\mu_{k}$ be sufficiently small, by the above we have
\begin{align*}
    \|\eta_{\mu_k}\|_{H^2}\leq\alpha_{\mu_k}\|Q\|_{H^2}+\alpha_{\mu_k}\mu_k C(p)+C\mu_k^{\frac{4}{3}}\alpha_{\mu_k}.
\end{align*}
Hence, $ \|\eta_{\mu_k}\|_{H^2}\rightarrow0$. This is a contradiction and the proof of this lemma is complete.
\end{proof}
Next lemma we shall give the general non-degeneracy property.
\begin{lemma}\label{nondegeneracy}
For the linear operator $L_{+,\mu}$ be given by \eqref{linear:operatordouble}, we have 
\begin{align*}
    \ker L_{+,\mu}=span\{\nabla Q_{\mu}\}.
\end{align*}
\end{lemma}
\begin{proof}
By the similar argument as \cite{L2009APDE}, we can obtain this lemma. For the reader's convenience, we will give the detail in the Appendix \ref{appendixnondegeneracy}.
\end{proof}

\section{Construction the approximate profile}

In this section, we aim to construct the approximate blowup profile $R_{\mathcal{P}}$ with the parameter $b,\alpha$. For a sufficiently regular function $f:\mathbb{R}^3\rightarrow \mathbb{C}$, we define the generator of $L^2$ scaling given by
\begin{align}\notag
\Lambda f:=\frac{3}{2}f+x\cdot\nabla f.
\end{align}
Note that the operator $\Lambda$ is skew-adjoint on $L^2(\mathbb{R}^3)$, that is, we have $(\Lambda f,g)=-(f,\Lambda g)$.

By the elementary calculation, we have the following algebraic identities,  which very important in the this section.
\begin{align}\label{identities:algebraic}
    L_{+,\mu}\Lambda Q_{\mu}=-2Q_{\mu},~~
    L_{+,\mu}(\nabla Q_{\mu})=0,~~
    L_{-,\mu}(Q_{\mu})=0.
\end{align}
From \cite{Cazenave:book,MV2013JFA}, we can obtain that  $Q_{\mu}$ and its derivatives are exponentially decay:
\begin{align}\notag
|\nabla Q_{\mu}|+Q_{\mu}(x)\lesssim \frac{e^{-|x|}}{|x|},\,\,|x|\geq1.
\end{align}
By the  above section (see Theorem \ref{Theorem1}), we have the following properties
\begin{align}\label{solvability:conditiondou}
\begin{cases}
\forall g\in L^2(\mathbb{R}^3),~(g,\nabla Q)=0\,~\exists\, f_+\in L^2(\mathbb{R}^3),\,L_{+,\mu}f_+=g,\\
\forall g\in L^2(\mathbb{R}^3),~(g,Q_{\mu})=0,\,~\exists\, f_-\in L^2(\mathbb{R}^3),\,L_{-,\mu}f_-=g.
\end{cases}
\end{align}

To construct minimal mass blowup solutions for problem \eqref{equ1:double}, we first renormalize the flow
\begin{align}\notag
u(t,x)=\frac{1}{\lambda^{\frac{3}{2}}(t)}v\left(s,\frac{x-\alpha(t)}{\lambda(t)}\right)e^{i\gamma(t)},\,
\,\,\frac{ds}{dt}=\frac{1}{t^2},
\end{align}
which leads the renormalized equation:
\begin{equation}\label{equ:renormalizedD}
i\partial_sv+\Delta v-v+|v|^{\frac{4}{3}}v+
\mu A(v^2)v=i\frac{\lambda_s}{\lambda}\Lambda v+i\frac{\alpha_t}{\lambda}\cdot\nabla v+\tilde{\gamma}_sv,
\end{equation}
where we have defined
$\tilde{\gamma}_s=\gamma_s-1.$
Let the pseudo-conformal drift:
\begin{align}\notag
v=we^{-ib|y|^2/4},
\end{align}
which leads to the slowly modulated equation
\begin{align*}
&i\partial_sw+\Delta w-w+|w|^{\frac{4}{3}}w+
\mu A(w^2)w
+\frac{1}{4}(b_s+b^2)|y|^2w\notag\\
=&i\left(\frac{\lambda_s}{\lambda}+b\right)\Lambda w+b\left(b+\frac{\lambda_s}{\lambda}\right)\frac{|y|^2}{2}w+i\frac{\alpha_t}{\lambda}\left(\nabla w+\frac{iby}{2}\right)+
\tilde{\gamma}\tilde{w}.
\end{align*}
Since we look for blowup solutions, the parameter $\lambda(s)$ should converge to zero as $s\rightarrow\infty$. Therefore, we now proceed to the slow modulated ansatz construction as in \cite{KMR2009CPAM,RS2011JAMS,MRS2014Duke,LMR2016RMI,KLR2013ARMA}. We freeze the modulation equations
\begin{align}\notag
\frac{\lambda_s}{\lambda}=-b,\,\,\,\,\,
b_s=-b^2,~~~\frac{\alpha_s}{\lambda}=d,~~~d_s=-2bd.
\end{align}
To have more clear information about the behaviour of these remodulation functions, we note that
\begin{equation}\notag
    b(s) \sim  s^{-1}, \lambda \sim s^{-1}, |d| \sim s^{-2}, |\alpha| \sim s^{-2}
\end{equation}
We look for an approximate solution to \eqref{equ:renormalizedD} of the form
\begin{align}\notag
v(s,y)=R_{(b(s),d(s))}(y)
\end{align}
with an expansion
\begin{align}\label{eq:gex1}
    R_{\mathcal{P}}(y)=Q_{\mu}+\sum_{k+j\geq1}b^jd^k R_{j,k},~~\text{where}~~\mathcal{P}(s)=(b(s),d(s)),~~R_{j,k}=\left(T_{j,k}+iS_{j,k}\right).
\end{align}
This allows us to construct a high order approximation $R_{\mathcal{P}}$ solution to
\begin{align}\label{equ:approximate:double}
-ib^2\partial_bR_{\mathcal{P}}-ibd\cdot\partial_dQ_{\mathcal{P}}+\Delta R_{\mathcal{P}}&-R_{\mathcal{P}}+|R_{\mathcal{P}}|^{\frac{4}{3}}R_{\mathcal{P}}
+\mu A(R_{\mathcal{P}}^2)R_{\mathcal{P}}\notag\\
&+ib\Lambda R_{\mathcal{P}}-id\cdot\nabla R_{\mathcal{P}}=-\Psi_{\mathcal{P}},
\end{align}
where $\Psi_{\mathcal{P}}=\mathcal{O}(b^5+|d||\mathcal{P}|^2)$ is some small error term.
\begin{lemma}\label{lemma:approximatedouble}
(Approximate Blowup Profile) Let $\mathcal{P}=(b,d)\in\mathbb{R}\times\mathbb{R}^3$. There exists a smooth function $R_{\mathcal{P}}=R_{\mathcal{P}}(x)$ of the form 
\begin{align}\notag
R_{\mathcal{P}}(y)=&Q_{\mu}+ibS_{1,0}+id\cdot S_{0,1}+bd\cdot T_{1,1}+b^2T_{2,0}+d^2\cdot T_{0,2}+ib^3S_{3,0}+b^4T_{4,0}+ib^2d\cdot S_{2,1}
\end{align}
is a solution satisfies  \eqref{equ:approximate:double}. Here, the function $\{R_{k,l}\}_{0\leq k\leq 4,0\leq l\leq 2}$ satisfy the following regularity and decay bounds:
\begin{align*}
    &\|R_{\mathcal{P}}\|_{H^2}+\|\Lambda R_{\mathcal{P}}\|_{H^2}\lesssim 1,~~
    |R_{\mathcal{P}}(x)|+|\Lambda R_{\mathcal{P}}(x)|\lesssim e^{-|x|}.
\end{align*}
The remainder   $\Psi$ in \eqref{equ:approximate:double}   satisfies the estimate
\begin{align}\label{approximate:decay:dou}
\sup_{y\in\mathbb{R}^3}\left(|\Psi(y)|+|\nabla\Psi(y)|\right)\lesssim (b^5+|d|^2)e^{-c|y|},\,\,\text{where}\,\,0<c<1.
\end{align}
\end{lemma}
\begin{proof}

We shall use the complete expression \eqref{eq:gex1} and we divide the rest of the proof of this lemma as follows.

\textbf{Step 1.} Determining the functions $\{T_{j,k}, S_{j,k}\}$.
We discuss our ansatz for $R_{\mathcal{P}}$ to solve \eqref{equ:approximate:double} order by order.

\textbf{Order} $\mathcal{O}(1):$ Clearly, we have that
\begin{align}\notag
-\Delta Q_{\mu}+Q_{\mu}-Q_{\mu}^{\frac{4}{3}}Q_{\mu}-\mu A\left(Q_{\mu}^2\right)(y)Q_{\mu}=0.
\end{align}
Since $Q_{\mu}=Q_{\mu}(|x|)>0$ being the ground state solution.

\textbf{Order:} $\mathcal{O}(b)$: By the Taylor expansion, we have
\begin{align*}
|R_{\mathcal{P}}|^{\frac{4}{3}}R_{\mathcal{P}}=&|Q_{\mu}+b(T_{1,0}+iS_{1,0})|^{\frac{4}{3}}(Q_{\mu}+b(T_{1,0}+iS_{1,0}))\\
=&\left(Q_{\mu}^{\frac{4}{3}}+\frac{4}{3}bQ_{\mu}^{\frac{1}{3}}T_{1,0}\right)\left(Q_{\mu}+b(T_{1,0}+iS_{1,0})\right)+\mathcal{O}(b^2)\\
=&Q_{\mu}^{\frac{7}{3}}+\frac{7}{3} bQ_{\mu}^{\frac{4}{3}}T_{1,0}+bi Q_{\mu}^{\frac{4}{3}}S_{1,0}
+\mathcal{O}(b^2),
\end{align*}
and the non-local term
\begin{align*}
A\left(|R_{\mathcal{P}}|^2\right)R_{\mathcal{P}}=&A\left(|Q_{\mu}+b(T_{1,0}+iS_{1,0})|^2\right)(Q_{\mu}+b(T_{1,0}+iS_{1,0}))\\
=&A\left(Q_{\mu}^2\right)Q_{\mu}+2bA\left(|Q_{\mu}T_{1,0}|\right)Q_{\mu}+bA\left(Q_{\mu}^2\right)T_{1,0}+ibA\left(Q_{\mu}^2\right)S_{1,0}+\mathcal{O}(b^2).
\end{align*}
Hence, we can obtain the equation
\begin{align}\notag
\begin{cases}
L_{+,\mu}T_{1,0}=0,\\
L_{-,\mu}S_{1,0}=\Lambda Q_{\mu}.
\end{cases}
\end{align}
Hence, choosing $T_{1,0}=0$ and note that $\Lambda Q_{\mu}\perp \ker L_{-,\mu}$ due to the fact that $(\Lambda Q_{\mu}, Q_{\mu})=0$. By \eqref{solvability:conditiondou},   there exists a unique $S_{1,0}$ satisfies the equation.

\textbf{Order} $\mathcal{O}(d)$: By the similar expansion as above,
so we have
\begin{align}\notag
    \begin{cases}
    L_{+,\mu}T_{0,1}=0,\\
    L_{-,\mu}S_{0,1}=-\nabla Q_{\mu}.
    \end{cases}
\end{align}
Choosing $T_{0,1}=0$, since $(\nabla Q_{\mu}, Q_{\mu})=0$, then there exists a  solution $S_{0,1}\perp \ker L_{-,\mu}$.

\textbf{Order}\ $\mathcal{O}(bd)$:
We find that $R_{1,1}$ has to solve the equation
\begin{align*}
    \begin{cases}
    L_{+,\mu}T_{1,1}=S_{0,1}-\Lambda S_{0,1}+\nabla S_{1,0}+\frac{4}{3}S_{1,0}S_{0,1}Q_{\mu}^{\frac{1}{3}}+2\mu A(S_{1,0}S_{0,1})Q_{\mu},\\
    L_{-,\mu}S_{1,1}=0.
    \end{cases}
\end{align*}
Hence, { choosing $S_{1,1}=0$, we see that } the existence of solution to this equation is guaranteed if 
$$ S_{0,1}-\Lambda S_{0,1}+\nabla S_{1,0}+\frac{4}{3}S_{1,0}S_{0,1}Q_{\mu}^{\frac{1}{3}}+2\mu A(S_{1,0}S_{0,1})Q_{\mu} \perp Ker L_{+,\mu}^* $$
so Lemma \ref{nondegeneracy} and self - adjointness of $L_{+,\mu}$ require to check the following orthogonality condition
\begin{align}\label{construct:bdclaim}
    \left(S_{0,1}-\Lambda S_{0,1}+\nabla S_{1,0}+\frac{4}{3}S_{1,0}S_{0,1}Q_{\mu}^{\frac{1}{3}}+2\mu A(S_{1,0}S_{0,1})Q_{\mu},\nabla Q_{\mu}\right)=0.
\end{align}
To verify it we use the commutator formula $[\Lambda,\nabla]=-\nabla$ and integrating by parts, we find
\begin{align}\label{construct:bd1}
    -(\nabla Q_{\mu},\Lambda S_{0,1})=&(\Lambda\nabla Q,S_{0,1})=(\nabla\Lambda Q_{\mu},S_{0,1})-(\nabla Q_{\mu},S_{0,1})\notag\\
    =&(\nabla L_{-,\mu}S_{1,0},S_{0,1})-(\nabla Q{\mu},S_{0,1}).
\end{align}
Next, since $L_{-,\mu}$ is self-adjoint, we observe that for any $f\in L^2(\mathbb{R}^3)$, 
\begin{align}\label{construct:bd2}
    (\nabla L_{-,\mu}f,S_{0,1})+(\nabla Q_{\mu},\nabla f)=&-(L_{-,\mu}f,\nabla S_{0,1})-(L_{-,\mu}S_{0,1},\nabla f)=(f,[\nabla,L_{-,\mu}]S_{0,1})\notag\\
    =&-(f,(\nabla (Q_{\mu}^{\frac{4}{3}}+\mu A(Q_{\mu}^2)))\cdot S_{0,1})\notag\\
    =&-\left(f,\frac{4}{3}Q_{\mu}^{\frac{1}{3}}\nabla Q_{\mu}\cdot S_{0,1}+\mu\nabla A(Q_{\mu}^2)\cdot S_{0,1}\right)\notag\\
    =&-(\nabla Q_{\mu},\frac{4}{3}Q_{\mu}^{\frac{1}{3}}S_{0,1}f)-(f,2\mu A(Q_{\mu}\nabla Q_{\mu})\cdot S_{0,1})\notag\\
    =&-(\nabla Q_{\mu},\frac{4}{3}Q_{\mu}^{\frac{1}{3}}S_{0,1}f)-(\nabla Q_{\mu},2\mu A(fS_{0,1})Q_{\mu}).
\end{align}
Combining \eqref{construct:bd1} and \eqref{construct:bd2}, we conclude that \eqref{construct:bdclaim} holds. Hence there exists a  solution $T_{1,1}\perp\ker L_{+,\mu}$ and $T_{1,1} \in H^2.$

\textbf{Order} $\mathcal{O}(b^2):$ 
By the Taylor expansion, we can obtain the equation
\begin{align*}
\begin{cases}
    L_{+,\mu}T_{2,0}=\frac{2}{3}Q_{\mu}^{\frac{1}{3}}S_{1,0}^2+S_{1,0}-\Lambda S_{1,0}+\mu A(S_{1,0}^2)Q_{\mu},\\
    L_{-,\mu}S_{2,0}=0 , \ \ { \text{hence \ \ $S_{2,0}=0$.}}
\end{cases}
\end{align*}
The solvability condition reduce to 
\begin{align}\notag
    \left(\frac{2}{3}Q_{\mu}^{\frac{1}{3}}S_{1,0}^2+S_{1,0}-\Lambda S_{1,0}+\mu A(S_{1,0}^2)Q_{\mu},\nabla Q_{\mu}\right)=0.
\end{align}
{ Now we can use the simple observation that
$\left( f, \nabla g\right) =0 $ for any couple of radial $H^1$ functions, then we observe that
 $Q_{\mu}$ and $S_{1,0}$ are radial functions,so the orthogonality condition is true and  there exists $ T_{2,0}\perp \ker L_{+,\mu}$ that satisfies the equation.}
 
\textbf{Order} $\mathcal{O}(d^2):$
We have the following system
\begin{align*}
    \begin{cases}
    L_{+,\mu}T_{0,2}=\nabla  S_{0,1}+\frac{2}{3}S_{0,1}^2Q_{\mu}^{\frac{1}{3}}+\mu A(S_{0,1}^2)Q_{\mu},\\
    L_{-,\mu}S_{0,2}=0.
    \end{cases}
\end{align*}
The solvability conditions reads
\begin{align}\notag
    \left(\nabla S_{0,1}+\frac{2}{3}S_{0,1}^2Q_{\mu}^{\frac{1}{3}}+\mu A(S_{0,1}^2)Q_{\mu},\nabla Q_\mu\right)=0.
\end{align}
Obviously, this is true, since $Q_{\mu}$ is radial function and each component of $S_{0,1}$ is odd function in $x$. Hence, there exists a  $T_{0,2}\perp \ker L_{+,\mu}$.

\textbf{Order} $\mathcal{O}(b^3)$: 
By the Taylor expansion, we can obtain the equation
\begin{equation}\notag
\begin{cases}
    L_{+,\mu}T_{3,0}=0,\\
    L_{-,\mu}S_{3,0}=\frac{4}{3}Q_{\mu}^{\frac{1}{3}}T_{2,0}S_{1,0}+\frac{2}{3}Q_{\mu}^{-\frac{2}{3}}S_{1,0}^3+\Lambda T_{2,0}-2T_{2,0}+\mu A(S_{1,0}^2)S_{1,0}+2\mu A(Q_{\mu}T_{2,0})S_{1,0}.
\end{cases}
\end{equation}
Hence, choosing $T_{3,0}=0,$ we see that the solvability condition for $S_{3,0}$ is equivalent to
\begin{align}\label{construction:S3dou}
    &-2(Q_{\mu},T_{2,0})+(Q_{\mu},\Lambda T_{2,0})+\frac{4}{3}\left(Q_{\mu},Q_{\mu}^{\frac{1}{3}}T_{2,0}S_{1,0}\right)+\frac{2}{3}\left(Q_{\mu},Q_{\mu}^{-\frac{2}{3}}S_{1,0}^3\right)\notag\\
    &+\mu(Q_{\mu},A(S_{1,0}^2)S_{1,0})+2\mu A(Q_{\mu}T_{2,0})S_{1,0})=0,
\end{align}
where the functions $S_{1,0}$ and $T_{2,0}$ satisfy
\begin{align*}
    L_{-,\mu}S_{1,0}=&\Lambda Q_{\mu},~~
     L_{+,\mu}T_{2,0}=\frac{2}{3}Q_{\mu}^{\frac{1}{3}}S_{1,0}^2+S_{1,0}-\Lambda S_{1,0}+\mu A(S_{1,0}^2)Q_{\mu}.
\end{align*}
To see that \eqref{construction:S3dou} holds, we first note that
\begin{align*}
    &\text{The left-hand side of \eqref{construction:S3dou}}\\
    =&-2(Q_{\mu},T_{2,0})-(\Lambda Q_{\mu},T_{2,0})+\frac{4}{3}\left(T_{2,0},Q_{\mu}^{\frac{4}{3}}S_{1,0}\right)+\frac{2}{3}\left(Q_{\mu}^{\frac{1}{3}},S_{1,0}^3\right)\\
    &+\mu(Q_{\mu},A(S_{1,0}^2)S_{1,0})+2\mu(Q_{\mu},A(Q_{\mu}T_{2,0})S_{1,0})\\
    =&-2(Q_{\mu},T_{2,0})-(L_{-,\mu}S_{1,0},T_{2,0})+\frac{4}{3}(T_{2,0},Q_{\mu}^{\frac{4}{3}}S_{1,0})+\frac{2}{3}(Q_{\mu}^{\frac{1}{3}},S_{1,0}^3)\\
    &+\mu(Q_{\mu},A(S_{1,0}^2)S_{1,0})+2\mu(Q_{\mu},A(Q_{\mu}T_{2,0})S_{1,0})\\
    =&-2(Q_{\mu},T_{2,0})-(L_{+,\mu}S_{1,0},T_{2,0})+\frac{2}{3}\left(Q_{\mu}^{\frac{1}{3}},S_{1,0}^3\right)+\mu(Q_{\mu},A(S_{1,0}^2)S_{1,0})\\
    =&-2(Q_{\mu},T_{2,0})-(S_{1,0},S_{1,0})+(S_{1,0},\Lambda S_{1,0})-\frac{2}{3}(Q_{\mu}^{\frac{1}{3}},S_{1,0}^3)+\frac{2}{3}(Q_{\mu}^{\frac{1}{3}},S_{1,0}^3)\\
    &+\mu(Q_{\mu},A(S_{1,0}^2)S_{1,0})-\mu(S_{1,0},A(S_{1,0}^2)Q_{\mu})\\
    =&-2(Q_{\mu},T_{2,0})-(S_{1,0},S_{1,0}),
\end{align*}
where in the last step we used that $(S_{1,0},\Lambda S_{1,0})=0$, since $\Lambda^*=-\Lambda$. Thus it remains to show that
\begin{align}\label{construction:relation}
    -2(Q_{\mu},T_{2,0})=(S_{1,0},S_{1,0}).
\end{align}
Indeed, by using $L_{+,\mu}\Lambda Q_{\mu}=-2Q_{\mu}$ (see \eqref{identities:algebraic}) and the equations for $T_{2,0}$ and $S_{1,0}$ above, we deduce
\begin{align}\label{construction3dou}
  -2(Q_{\mu},T_{2,0})=&(L_{+,\mu}\Lambda Q_{\mu},T_{2,0})\notag\\
  =&\left(\Lambda Q_{\mu},\frac{2}{3}Q_{\mu}^{\frac{1}{3}}S_{1,0}^2+S_{1,0}-\Lambda S_{1,0}+\mu A(S_{1,0}^2)Q_{\mu}\right)\notag\\
  =&(L_{-,\mu}S_{1,0},S_{1,0})-(L_{-,\mu}S_{1,0},\Lambda S_{1,0})+\frac{2}{3}(\Lambda Q_{\mu},Q_{\mu}^{\frac{1}{3}}S_{1,0}^2)
  +\mu\left(\Lambda Q_{\mu},A(S_{1,0}^2)Q_{\mu}\right)\notag\\
  =&(S_{1,0},-\Delta S_{1,0})+(S_{1,0},S_{1,0})-(S_{1,0},Q_{\mu}^{\frac{4}{3}}S_{1,0})-\mu\left( A(Q_{\mu}^2)S_{1,0},S_{1,0}\right)\notag\\
  &-(L_{-,\mu}S_{1,0},\Lambda S_{1,0})+\frac{2}{3}(\Lambda Q_{\mu},Q_{\mu}^{\frac{1}{3}}S_{1,0}^2)+\mu\left(\Lambda Q_{\mu},A(S_{1,0}^2)Q_{\mu}\right).
\end{align}
Next, we have the commutator formula
$(L_{-,\mu}f,\Lambda f)=\frac{1}{2}(f,[L_{-,\mu},\Lambda]f),$
which show that
\begin{align}\label{construction4dou}
    &(L_{-,\mu}S_{1,0},\Lambda S_{1,0})\notag\\
    =&\frac{1}{2}(S_{1,0},[L_{-,\mu},\Lambda]S_{1,0})\notag\\
    =&\frac{1}{2}(S_{1,0},[-\Delta,\Lambda]S_{1,0})-\frac{1}{2}\left(S_{1,0},[Q_{\mu}^{\frac{4}{3}},\Lambda]S_{1,0}\right)-\frac{1}{2}\mu\left(S_{1,0},[A(|Q_{\mu}|^2),\Lambda]S_{1,0}\right)\notag\\
    =&(S_{1,0},-\Delta S_{1,0})+\frac{2}{3}\left(S_{1,0},(x\cdot\nabla Q_{\mu})Q_{\mu}^{\frac{1}{3}}S_{1,0}\right)-\frac{1}{2}\mu\left(S_{1,0},[A(|Q_{\mu}|^2),\Lambda]S_{1,0}\right)
\end{align}
using that $[-\Delta,\Lambda]=-2\Delta$ holds. Moreover, we have the pointwise identity
\begin{align}\label{construction5dou}
    -(x\cdot\nabla Q_{\mu})Q_{\mu}^{\frac{1}{3}}+Q_{\mu}^{\frac{1}{3}}\Lambda Q_{\mu}=\frac{3}{2}Q^{\frac{4}{3}}.
\end{align}
Furthermore, we have
\begin{align}\label{construction6dou}
(S_{1,0},[A(|Q_{\mu}|^2),\Lambda]S_{1,0})=&(S_{1,0},x\cdot\nabla A(Q_{\mu}^2)S_{1,0})=(S_{1,0},(x\cdot\nabla(-\Delta)^{-\frac{1}{2}}Q_{\mu}^2)S_{1,0})\notag\\
=&(S_{1,0},((-\Delta)^{-\frac{1}{2}}x\cdot\nabla Q_{\mu}^2)S_{1,0})+(S_{1,0},(-\Delta)^{-\frac{1}{2}}Q_{\mu}^2S_{1,0})\notag\\
=&2(S_{1,0},(|x|^{-2}*(Q_{\mu}x\cdot\nabla Q_{\mu})S_{1,0})+(S_{1,0},A(Q_{\mu}^2)S_{1,0})\notag\\
=&2(\Lambda Q_{\mu},A(S_{1,0}^2)Q_{\mu}).
\end{align}
Here we also used the commutator formula
$[(-\Delta)^{-\frac{1}{2}},x\cdot\nabla]=-(-\Delta)^{-\frac{1}{2}}.$
Now if we insert \eqref{construction4dou}, \eqref{construction5dou} and \eqref{construction6dou} into \eqref{construction3dou}, we can obtain the desire relation \eqref{construction:relation}, and thus the solvability condition \eqref{construction:S3dou} holds as well.

\textbf{Order} $\mathcal{O}(b^4)$:
By the Taylor expansion, we have the  following equation
\begin{equation*}
    \begin{cases}
    L_{+,\mu}T_{4,0}=-\frac{4}{3}Q_{\mu}^{\frac{1}{3}}S_{1,0}S_{3,0}+\frac{14}{9}Q_{\mu}^{\frac{1}{3}}T_{2,0}^2-\frac{1}{9}(Q_{\mu}^{-\frac{5}{3}}S_{1,0}^4+4Q_{\mu}^{-\frac{2}{3}}T_{2,0}S_{1,0}^2)+3S_{3,0}-\Lambda S_{3,0}+\mu B_1,\\
    L_{-,\mu}S_{4,0}=0,
    \end{cases}
\end{equation*}
where $B_1=A(2Q_{\mu}T_{2,0}+S_{1,0}^2)T_{2,0}+A(T_{2,0}^2+2S_{1,0}S_{3,0})Q_{\mu}.$
By the above construction, we can obtain that $S_{1,0}$, $S_{3,0}$, $T_{2,0}$ are radial functions. Hence, there  exists $T_{4,0}\perp\ker L_{+,\mu}$ satisfies the above equation.

\textbf{Order} $\mathcal{O}(b^2d)$: 
By the calculation, we can obtain the following equation 
\begin{align*}
    \begin{cases}
    L_{+,\mu}T_{2,1}=0,\\
    L_{-,\mu}S_{2,1}=\frac{4}{3}Q_{\mu}^{\frac{1}{3}}T_{1,1}S_{1,0}+\frac{4}{3}Q_{\mu}^{\frac{1}{3}}T_{2,0}S_{0,1}+2Q_{\mu}^{-\frac{2}{3}}S_{1,0}^2S_{0,1}
    -3T_{1,1}+\Lambda T_{1,1}-\nabla T_{2,0}+\mu B_1,
    \end{cases}
\end{align*}
where $B_2=A(2Q_{\mu}T_{2,0}+S_{1,0}^2)S_{0,1}+A(S_{1,0}S_{0,1})S_{1,0}$.
Since $Q_{\mu}$, $S_{1,0}$, $T_{2,0}$ are radial functions and each components in $S_{0,1}$, $T_{1,1}$ are odd functions. Hence, there is $S_{2,1}\perp \ker L_{-,\mu}$.
\begin{lemma}\label{lemma1decaydou}
Let $f,g\in L^2(\mathbb{R}^3)$  and suppose that $f\perp Q_{\mu}$, $g\perp\nabla Q_{\mu}$. Then we have the regularity and decay estimate
\begin{align*}
 \|L_{-,\mu}^{-1}f\|_{H^{2}}\lesssim\|f\|_{L^2},\,\,
    \|L_{+,\mu}^{-1}g\|_{H^{2}}\lesssim\|g\|_{L^2},\\
    \|e^{c|x|}L_{-,\mu}^{-1}f\|_{H^2}\lesssim\|e^{c|y|}f\|_{L^2},\,\,
    \|e^{c|x|}L_{+,\mu}^{-1}g\|_{H^2}\lesssim\|e^{c|y|}g\|_{L^2}\,\,\,\text{where}\,\,0<c<1.
\end{align*}
\end{lemma}
\begin{proof}
It suffices to prove the lemma for $L_{-,\mu}^{-1}$, since the estimates for $L_{+,\mu}^{-1}$ follow in the same fashion.
To show the decay estimate, we argue as follows. Assume that $\|e^{c|y|}f\|_{L^2}<+\infty$, because otherwise there is nothing to prove. Let $u=L_{-,\mu}^{-1}f$, and rewrite the equation satisfied by $u$ in resolvent form:
\begin{align*}
   (-\Delta+1) u= Q_{\mu}^{\frac{4}{3}}u+\mu A(Q_{\mu}^2)u+ f.
\end{align*}
In fact, by the elliptic regularity theorem (see \cite{Cazenave:book,MV2013JFA}), we have $Q_{\mu}\in W^{2,p}(\mathbb{R}^3)$, where $p\geq1$. And $Q_{\mu}(|x|)=\frac{e^{-|x|}}{|x|}(c_0+\mathcal{O}(\frac{1}{|x|}))$ as $|x|\geq R$. Hence, we have
\begin{align*}
    \|e^{c|y|}u\|_{H^2(|y|\geq R)}\sim&\|(-\Delta+1)(e^{c|y|}u)\|_{L^2(|y|\geq R)}\lesssim\|e^{c|y|}(-\Delta+1)u\|_{L^2(|y|\geq R)}\\
    =&\|e^{c|y|}Q_{\mu}^{\frac{4}{3}}u\|_{L^2(|y|\geq R)}+\mu\|e^{c|y|}A(Q_{\mu}^2)u\|_{L^2(|y|\geq R)}+\|e^{c|y|}f\|_{L^2(|y|\geq R)}.
\end{align*}
From this inequality, we can deduce that
\begin{align*}
    \|e^{c|y|}u\|_{H^2(|y|\geq R)}\lesssim \|e^{c|y|}f\|_{L^2}.
\end{align*}
On the other hand, we have
\begin{align*}
    \|e^{c|y|}u\|_{H^2(|y|\leq R)}\lesssim\|u\|_{H^2(|y|\leq R)}\lesssim\|f\|_{L^2(|y|\leq R)}\lesssim\|e^{c|y|}f\|_{L^2}.
\end{align*}
From this, we can obtain the desired result.
\end{proof}

\textbf{Step 2.} Now, we turn back to the prove \eqref{approximate:decay:dou}.
By the above lemma \ref{lemma1decaydou}  and the similar argument as \cite{LMR2016RMI}, we can easily obtain \eqref{approximate:decay:dou}. For the regularity and decay estimate, this can be obtain by the following lemma \ref{lemmlast}.
\end{proof}
The following lemma show that the approximate profile $Q_{\mathcal{P}}$ is well-define.
\begin{lemma}\label{lemmlast}
By the definition of $R_{\mathcal{P}}$, we have
\begin{align*}
    |R_{\mathcal{P}}|\lesssim Q_{\mu}.
\end{align*}
\end{lemma}
\begin{proof}
Since $R_{\mathcal{P}}=Q_{\mu}+\sum_{0\leq k\leq 4,0\leq l\leq2}R_{k,l}$, we need to prove
\begin{align}\notag
    \left|\frac{R_k}{Q_{\mu}}\right|\lesssim1.
\end{align}
For any $f_k\in L^2(\mathbb{R}^3)$ and $|f_k|\lesssim e^{-c|x|}$, $c\geq1$, we assume that
\begin{align*}
    L_{-,\mu}F_{k,l}=f_{k,l},\,\,\,\,\text{or}\,\,L_{+,\mu}F_{k,l}=f_{k,l}.
\end{align*}
In other words, we have to prove
\begin{align*}
    L_{+,\mu}^{-1}: L_{G^c}^{\infty}:=\left\{f:\,\frac{f}{G^c}\in L^{\infty}\right\}\rightarrow L_{G^1}^{\infty}=:\left\{f:\,\frac{f}{G}\lesssim1\right\},
\end{align*}
where $G(x)=\frac{e^{-|x|}}{|x|}$ and $c\geq1$.

\textbf{Step~1}: We first prove the following holds:
\begin{align*}
    G^{-1}(1-\Delta)^{-1} G^c: L^{\infty} \rightarrow L^{\infty}, \ c \geq 1.
\end{align*}
Indeed, we have
\begin{align*}
    \left|G*f_k\right|\lesssim\left|\int\frac{e^{-|x-y|}}{|x-y|}\frac{e^{-c|y|}}{|y|^c}\frac{f_k}{G^c}\right|\lesssim G(x).
\end{align*}
\textbf{Step~2}: Let $L_+ = -\Delta+1-\frac{7}{3}Q^{\frac{4}{3}}.$
Our goal is to show that
\begin{align*}
    G^{-1}(L_+)^{-1} G^c: L^{\infty} \rightarrow L^{\infty}, \ c \geq 1.
\end{align*}

Indeed, let $L_+u=g$, i.e.,
\begin{align} \label{eq:st2}
    (-\Delta+1)u=\frac{7}{3}Q_{\mu}^{\frac{4}{3}}u+g.
\end{align}
Using the $H^2$ estimate for $L_+$ we can write
$$  \|e^{d|x|}L_{+}^{-1}g\|_{H^2}\lesssim\|e^{d|y|}g\|_{L^2}, ~~0 \leq d <1.$$
So from the Strauss estimate we get
$$\|e^{c|x|}u\|_{L^\infty(|x|\geq 1)} \lesssim 1.$$
Plugging this estimate in the right hand side of \eqref{eq:st2} and using Step 1,  we get $|u|\lesssim G.$

\textbf{Step~3}: The operator $K: f\rightarrow A(Q_{\mu}f)$ maps $ L^{\infty}_G$ into  $ L^{\infty}_{G^{1-\epsilon}}$.
Since $\mu>0$ is small enough, then we deduce that
\begin{align*}
    G^{-1}(L_{+,\mu})^{-1} G^c: L^{\infty} \rightarrow L^{\infty}, \ c \geq 1.
\end{align*}
By the similar argument, we can obtain that $L_{-,\mu}^{-1}$ also satisfies this property. The proof of this lemma is now complete.
\end{proof}
\begin{remark}\label{remark1}
 (i) Note that $L_{-,\mu}>0$ on $Q_{\mu}^{\perp}$ and we have $S_{1,0}\perp Q_{\mu}$, $S_{0,1}\perp Q_{\mu}$.

 (ii) The proof of  lemma \ref{lemma:approximatedouble} actually show that $R_{\mathcal{P}}$ satisfy
\begin{align*}
	R_{\mathcal{P}}&=(Q_{\mu}+bd\cdot T_{1,1}+b^2T_{2,0}+d^2\cdot T_{0,2}+b^4T_4)+i(bS_{1,0}+d\cdot S_{0,1}+b^3S_{3,0}+b^2d\cdot S_{2,1})\\
	&=R_1+iR_2.
\end{align*}
\end{remark}
Let us compute the $L^2$-norm and energy of $R_{\mathcal{P}}$, which will appear as important quantities in the analysis.
\begin{lemma}
The mass, energy and momentum of $R_{\mathcal{P}}$ satisfy:
\begin{align*}
    &\int|R_{\mathcal{P}}|^2=\int|Q_{\mu}|^2+\mathcal{O}(b^4+|d|^2+|d\mathcal{P}|^2 ),\\
    &E_{\mu}(R_{\mathcal{P}})=b^2e_{\mu}+\mathcal{O}(b^4+|d|^2+|d\mathcal{P}|^2),\\
    &P(R_{\mathcal{P}})=p_{\mu}d+\mathcal{O}(b^4+|d|^2+|d\mathcal{P}|^2),
\end{align*}
where $e_{\mu}=\frac{1}{2}(L_{-,\mu}S_{1,0},S_{1,0})>0$ and $p_{\mu}=2(L_{-,\mu}S_{0,1},S_{0,1})$ are  constants and $S_{1,0}$, $S_{0,1}$ satisfy $L_{-,\mu}S_{1,0}=\Lambda Q_{\mu}$, $L_{-,\mu}S_{0,1}=-\nabla Q_{\mu}$, respectively.
\end{lemma}
\begin{proof}
From the lemma \ref{lemma:approximatedouble} and the Remark \ref{remark1}, we deduce that
\begin{align*}
    \int|R_{\mathcal{P}}|^2=&\int Q_{\mu}^2+b^2(S_{1,0},S_{1,0})+2b^2(Q_{\mu},T_{2,0})+\mathcal{O}(b^4+|d|^2+|d\mathcal{P}|^2)\\
    =&\int Q_{\mu}^2+\mathcal{O}(b^4+|d|^2+|d\mathcal{P}|^2),
\end{align*}
where we use the relation $(S_{1,0},S_{1,0})+(Q_{\mu},T_{2,0})=0$, see \eqref{construction:relation}.

To calculate the expansion of the energy, we first recall that $E_{\mu}(Q_{\mu})=0$, this can be obtained by the Pohozaev identity \eqref{Pohozaev:identitydou} and the equation \eqref{equ:ell1}. Moreover, from the remark \ref{remark1}, we have $(Q_{\mu},S_{1,0})=0$, we obtain
\begin{align*}
    E_{\mu}(R_{\mathcal{P}})
    =&b^2\Bigg[\left(T_{2,0},-\Delta Q_{\mu}-Q_{\mu}^{\frac{7}{3}}-\mu A(Q_{\mu}^2)Q_{\mu}\right)+\frac{1}{2}(S_{1,0},-\Delta S_{1,0})-\frac{1}{2}b^2(Q_{\mu}^{\frac{4}{3}},S_{1,0}^2)\\
    &-\mu\frac{1}{2}\int A(Q_{\mu}^2)S_{1,0}^2\Bigg]+\mathcal{O}(b^4+|d|^2+|d\mathcal{P}|^2)\\
    =&b^2\Bigg[-(T_{2,0}, Q_{\mu})+\frac{1}{2}(S_{1,0},-\Delta S_{1,0})-\frac{1}{2}(Q_{\mu}^{\frac{4}{3}},S_{1,0}^2)-\mu\frac{1}{2}\int A(Q_{\mu}^2)S_{1,0}^2\Bigg]\\
    &+\mathcal{O}(b^4+|d|^2+|d\mathcal{P}|^2)\\
   =&b^2\frac{1}{2}\Bigg[(S_{1,0},S_{1,0})+(S_{1,0},-\Delta S_{1,0})-(Q_{\mu}^{\frac{4}{3}},S_{1,0}^2)-\mu\int A(Q_{\mu}^2)S_{1,0}^2\Bigg]+\mathcal{O}(b^4+|d|^2+|d\mathcal{P}|^2)\\
   =&b^2\frac{1}{2}(L_{-,\mu}S_{1,0},S_{1,0})+\mathcal{O}(b^4+|d|^2+|d\mathcal{P}|^2).
\end{align*}
For the linear momentum functional, we notice that $P(f)=2\int f_1\nabla f_2$ for $f=f_1+if_2$. Hence \begin{align*}
    P(R_{\mathcal{P}})=&2\int bQ_{\mu}\nabla S_{1,0}+2d\int Q_{\mu}\nabla S_{0,1}+2\int b^2d\cdot T_{1,1}\nabla S_{1,0}+2\int b^3T_{2,0}\nabla S_{1,0}\\
    &+\mathcal{O}(b^4+|d|^2+|d\mathcal{P}|^2)\\
    =&2d(L_{-,\mu}S_{0,1},S_{0,1})+\mathcal{O}(b^4+|d|^2+|d\mathcal{P}|^2).
\end{align*}
Here we used the fact that $L_{-,\mu}S_{0,1}=-\nabla Q_{\mu}$ and $S_{1,0}$, $T_{1,1}$, $T_{2,0}$ are radial function. The proof of this lemma is now complete.
\end{proof}

\section{Energy Estimates}
\subsection{Nonlinear decomposition of the wave and modulation equations}
Let $u(t)\in H^1(\mathbb{R}^3)$ be a solution of equation \eqref{equ1:double} on some time interval $[t_0,t_1]$ with $t_1<0$. Assume that $u(t)$ admits a geometrical decomposition of the form
\begin{align}\label{decom:solutiondou}
    u(t,x)=\frac{1}{\lambda^{\frac{3}{2}}(t)}\big[R_{\mathcal{P}}+\epsilon\big]\left(t,\frac{x-\alpha(t)}{\lambda(t)}\right)e^{i\gamma(t)},
\end{align}
with a  uniform smallness bound on $[t_0,t_1]$:
\begin{align}\notag
    b^2(t)+|d(t)|+\|\epsilon(t)\|_{H^1}^2\lesssim\lambda^2(t)\ll1.
\end{align}
Moreover, we assume that $u(t)$ has almost critical mass in the sense: $\forall\,t\in[t_0,t_1]$,
\begin{align}\notag
    \left|\|u(t)\|_{L^2}^2-\|Q_{\mu}\|_{L^2}^2\right|\lesssim \lambda^{4}(t).
\end{align}
From a standard modulation argument, see e.g.\cite{RS2011JAMS,MR2005Ann}, the uniqueness of the nonlinear decomposition \eqref{decom:solutiondou} may be ensured by imposing a suitable set of orthogonality condition on $\epsilon=\epsilon_1+i\epsilon_2\in H^1(\mathbb{R}^3)$; namely,
\begin{align}\label{orthogonality1dou}
\begin{array}{c}
    (\epsilon_2,\Lambda R_1)-(\epsilon_1,\Lambda R_2)=0,\\
    (\epsilon_2,\partial_bR_1)-(\epsilon_1,\partial_bR_2)=0,\\
    (\epsilon_2,\rho_1)-(\epsilon_1,\rho_2)=0,\\
    (\epsilon_2,\nabla R_1)-(\epsilon_1,\nabla R_2)=0,\\
    (\epsilon_2,\partial_dR_1)-(\epsilon_1,\partial_dR_2)=0,
\end{array}
\end{align}
where $R_1$ and $R_2$ are given by Remark \ref{remark1} and $\rho=\rho_1+i\rho_2$ is the unique  function defined by
\begin{align}\notag
\begin{cases}
    L_{+,\mu}\rho_1=S_{1,0},\\
    L_{-,\mu}\rho_2=\frac{4}{3}bQ_{\mu}^{\frac{1}{3}}S_{1,0}\rho_1+b\Lambda\rho_1-2bT_{2,0}+\mu \left(2A(Q_{\mu}\rho_1)S_{1,0}\right)\\
    ~~~~+\frac{4}{3}dQ_{\mu}^{\frac{1}{3}}S_{0,1}\rho_1+d\cdot\nabla\rho_1+d \cdot T_{1,1}+2\mu A(Q_{\mu}\rho_1)S_{0,1}.
\end{cases}
\end{align}
By \eqref{solvability:conditiondou}, $\rho_1$ is well-defined. Moreover, it is easy to see that the right-hand side in the equation  for $\rho_2$ is perpendicular to $Q_{\mu}$. Hence, $\rho_2$ is well-defined, too.
The orthogonality conditions \eqref{orthogonality1dou} correspond exactly in the cases $\mathcal{P}=0$ to the null space of the linearized operator close to $Q_{\mu}$ see \eqref{identities:algebraic}. From a standard argument, the obtained modulation parameters are $C^1$ functions of time; see \cite{RS2011JAMS,MR2005Ann},for related statements. Let
\begin{align}\notag
    s(t)=\int_{t_0}^{t_1}\frac{d\tau}{\lambda^2(\tau)},
\end{align}
be the rescaled time. Then,  for $s\in[s_0,s_1]$, the function $\epsilon$ satisfies the system
\begin{align}\label{equ:mod1}
    &(b_s+b^2)\partial_b R_1+(d_s+bd)\cdot\partial_dR_1+\partial_s\epsilon_1-M_2(\epsilon)+b\Lambda\epsilon_1-d\cdot\nabla\epsilon_1\notag\\
    =&\left(\frac{\lambda_s}{\lambda}+b\right)(\Lambda R_1+\Lambda\epsilon_1)+\left(\frac{\alpha_s}{\lambda}-d\right)\cdot(\nabla R_1+\nabla\epsilon_1)+\tilde{\gamma}_s(R_2+\epsilon_2)+\Im{\Phi_b}-P_2(\epsilon),\\\label{equ:mod2}
    &(b_s+b^2)\partial_b R_2+(d_s+bd)\cdot\partial_dR_2+\partial_s\epsilon_2+M_1(\epsilon)+b\Lambda\epsilon_2-d\cdot\nabla\epsilon_2\notag\\
    =&\left(\frac{\lambda_s}{\lambda}+b\right)(\Lambda R_2+\Lambda\epsilon_2)\left(\frac{\alpha_s}{\lambda}-d\right)\cdot(\nabla R_2+\nabla\epsilon_2)-\tilde{\gamma}_s(R_1+\epsilon_1)-\Re{\Phi_b}+P_1(\epsilon).
\end{align}
Here $\Phi_b$ denotes the error term and $M=(M_1,M_2)$ are small deformations of the linearized operator $L=(L_{+,\mu},L_{-,\mu})$ close to $Q_{\mu}$:
\begin{align}\label{define:M1dou}
    M_1(\epsilon)=-\Delta\epsilon_1+\epsilon_1-|R_{\mathcal{P}}|^{\frac{4}{3}}\epsilon_1-\frac{4}{3}|R_{\mathcal{P}}|^{-\frac{2}{3}}(R_1R_2\epsilon_2+R_1^2\epsilon_1)-\mu D_1(\epsilon),\\\label{define:M2dou}
    M_2(\epsilon)=-\Delta\epsilon_2+\epsilon_2-|R_{\mathcal{P}}|^{\frac{4}{3}}\epsilon_2-\frac{4}{3}|R_{\mathcal{P}}|^{-\frac{2}{3}}(R_1R_2\epsilon_1+R_2^2\epsilon_2)-\mu D_2(\epsilon),
\end{align}
where $D_1$ and $D_2$ are given by
\begin{align*}
    D_1{\epsilon}=A(2R_1\epsilon_1-2R_2\epsilon_2)R_1+A(R_1^2-R_2^2)\epsilon_1,~
    D_2{\epsilon}=A(2R_1\epsilon_1-2R_2\epsilon_2)R_2+A(R_1^2-R_2^2)\epsilon_2.
\end{align*}
In addition, $P_1(\epsilon)$ and $P_2(\epsilon)$ are the high order terms respect to $\epsilon$.

Let us collect the standard preliminary estimates on this decomposition which rely on the conservation laws and the explicit choice of orthogonality conditions.
\begin{lemma}\label{lemma:modestimate}
(\textbf{Preliminary estimates on the decomposition.}) For $t\in[t_0,t_1]$, it holds that

1. Energy and Momentum bound:
\begin{align}\notag
    b^2+|d|+\|\epsilon\|_{H^1(\mathbb{R}^3)}^2\lesssim \lambda^2(|E_{0,\mu}|+|P_{0,\mu}|)+\mathcal{O}(\lambda^{4}+b^4+|d|^2+|d\mathcal{P}|^2).
\end{align}
Here $E_{0,\mu}=E_{\mu}(u_0)$ and $P_{0,\mu}=P_{\mu}(u_0)$ denote the conserved energy and momentum of $u=u(t,x)$, respectively.

2. Control of the geometrical parameters: Let the vector of modulation equations be
\begin{align}\notag
Mod(t)=\left(b_s+b^2,\frac{\lambda_s}{\lambda}+b,\tilde{\gamma_s},\frac{\alpha_s}{\lambda}-d,d_s+bd\right).
\end{align}
Then the modulation equations are to leading order:
\begin{align}\notag
    |Mod(t)|\lesssim\lambda^4+b^4+|d|^2+|d\mathcal{P}|^2+b^2\|\epsilon\|_{L^2}+\|\epsilon\|_{L^2}^2+\|\epsilon\|_{H^1}^3,
\end{align}
with the improvement
\begin{align}\notag
    \left|\frac{\lambda_s}{\lambda}+b\right|\lesssim b^5+b^2\|\epsilon\|_{L^2}+\|\epsilon\|_{L^2}^2+\|\epsilon\|_{H^1}^3.
\end{align}
\end{lemma}
\begin{proof}
\textbf{Step 1.} By the similar argument as \cite{RS2011JAMS,GL2021CPDE,GL2021JFA,KLR2013ARMA}, we can obtain the energy and momentum estimates.

\textbf{Step 2.} Estimate the modulation parameters.

1. \textbf{Inner products.} We compute the inner products needed to compute the law of the parameters from the $R_{\mathcal{P}}$, where $M_1,M_2$ are given by \eqref{define:M1dou} and \eqref{define:M2dou}, respectively. The following estimates hold.
\begin{align}\label{estimateM1dou}
    &(M_2(\epsilon)-b\Lambda \epsilon_1+d\cdot\nabla \epsilon_1,\Lambda R_2)+(M_1(\epsilon)+b\Lambda \epsilon_2-d\cdot\nabla \epsilon_1,\Lambda R_1)=-\Re(\epsilon,R_{\mathcal{P}})+\mathcal{O}(\mathcal{P}^2\|\epsilon\|_{L^2}),\\\label{estimateM2dou}
   &(M_2(\epsilon)-b\Lambda \epsilon_1+d\cdot\nabla \epsilon_1,\partial_bR_2)+(M_1(\epsilon)+b\Lambda \epsilon_2-d\cdot\nabla \epsilon_1,\partial_bR_1)=\mathcal{O}(\mathcal{P}^2\|\epsilon\|_{L^2}),\\\label{estimateM3dou}
   &(M_2(\epsilon)-b\Lambda \epsilon_1+d\cdot\nabla \epsilon_1,\rho_2)+(M_1(\epsilon)+b\Lambda \epsilon_2-d\cdot\nabla \epsilon_1,\rho_1)=\mathcal{O}(\mathcal{P}^2\|\epsilon\|_{L^2}),\\\label{estimateM4dou}
   &(M_2(\epsilon)-b\Lambda \epsilon_1+d\cdot\nabla \epsilon_1,\nabla R_2)+(M_1(\epsilon)+b\Lambda \epsilon_2-d\cdot\nabla \epsilon_1,\nabla R_1)=\mathcal{O}(\mathcal{P}^2\|\epsilon\|_{L^2}),\\\label{estimateM5dou}
   &(M_2(\epsilon)-b\Lambda \epsilon_1+d\cdot\nabla \epsilon_1,\partial_dR_2)+(M_1(\epsilon)+b\Lambda \epsilon_2-d\cdot\nabla \epsilon_1,\partial_d R_1)=\mathcal{O}(\mathcal{P}^2\|\epsilon\|_{L^2}).
\end{align}
We notice that the identity
\begin{align}\label{identityS1dou}
    &L_{-,\mu}\Lambda S_{1,0}=-2(S_{1,0}-\Lambda Q_{\mu})+\frac{4}{3}(\Lambda Q_{\mu})Q_{\mu}^{\frac{1}{3}}S_{1,0}+\Lambda^2Q_{\mu}+2\mu A(\Lambda Q_{\mu}Q_{\mu})S_{1,0},\\\notag
    &L_{-,\mu}\Lambda S_{0,1}=-S_{0,1}-\nabla Q_{\mu}+\frac{4}{3}(\Lambda Q_{\mu})Q_{\mu}^{\frac{1}{3}}S_{0,1}+2\mu A(\Lambda Q_{\mu}Q_{\mu})S_{0,1}-\Lambda\nabla Q_{\mu}.
\end{align}
This two identities can deduce from  $L_{-,\mu}S_{1,0}=\Lambda Q_{\mu}$ and  $L_{-,\mu}S_{0,1}=-\nabla Q_{\mu}$. 
Next, recall that
\begin{align*}
    \Lambda R_1=\Lambda Q_{\mu}+\mathcal{O}(\mathcal{P}^2),\,\,\Lambda R_2=b\Lambda S_{1,0}+d\Lambda S_{0,1}+\mathcal{O}(\mathcal{P}^2).
\end{align*}
Using the equality \eqref{identityS1dou} and \eqref{identities:algebraic}, we find that
\begin{align*}
    &\text{Left-hand side of \eqref{estimateM1dou}}\\
    =&(\epsilon_1,L_{+,\mu}\Lambda Q_{\mu})+b(\epsilon_2,L_{-,\mu}\Lambda S_{1,0})+d\cdot(\epsilon_2,L_{-,\mu}\Lambda S_{0,1})-\frac{4}{3}b(Q_{\mu}^{\frac{1}{3}}S_{1,0}\epsilon_2,\Lambda Q_{\mu})\\
    &-\frac{4}{3}d\cdot(Q_{\mu}^{\frac{1}{3}}S_{0,1}\epsilon_2,\Lambda Q_{\mu})-b(\epsilon_2,\Lambda^2Q_{\mu})-d\cdot(\epsilon_2,\nabla\Lambda Q_{\mu})\\
    &-2b (A(Q_{\mu}\Lambda Q_{\mu})S_{1,0},\epsilon_2)-2d\cdot(A(Q_{\mu}\Lambda Q_{\mu})S_{0,1},\epsilon_2)+\mathcal{O}(\mathcal{P}^2\|\epsilon\|_{L^2})\\
     =&-2(\epsilon_1,Q_{\mu})-b(\epsilon_2,S_{1,0})+b(\epsilon_2,\Lambda Q_{\mu})-d\cdot(\epsilon_2,S_{0,1})-d\cdot(\epsilon_2,\nabla Q_{\mu})+\mathcal{O}(\mathcal{P}^2\|\epsilon\|_{L^2})\\
    =&-2\Re{(\epsilon,R_{\mathcal{P}})}+\mathcal{O}(\mathcal{P}^2\|\epsilon\|_{L^2}).
\end{align*}
Here, we also used that $b(\epsilon_2,\Lambda Q_{\mu})=\mathcal{O}(\mathcal{P}^2\|\epsilon\|_{L^2})$, $d\cdot(\epsilon_2,\nabla Q_{\mu})=\mathcal{O}(\mathcal{P}^2\|\epsilon\|_{L^2})$, which follows from the orthogonality conditions \eqref{orthogonality1dou}. This completes the proof \eqref{estimateM1dou}.

\textbf{Estimate \eqref{estimateM2dou}.} From the lemma \ref{lemma:approximatedouble} we recall that
\begin{align*}
    \partial_bR_1=2bT_{2,0}+d\cdot T_{1,1}+\mathcal{O}(\mathcal{P}^2),\,\,\,\partial_bR_2=S_{1,0}+\mathcal{O}(\mathcal{P}^2),
\end{align*}
where
\begin{align*}
    L_{+,\mu}T_{2,0}=&\frac{2}{3}Q_{\mu}^{\frac{1}{3}}S_{1,0}^2+S_{1,0}-\Lambda S_{1,0}+\mu A(S_{1,0}^2)Q_{\mu},\\
     L_{+,\mu}T_{1,1}=&S_{0,1}-\Lambda S_{0,1}+\nabla S_{1,0}+\frac{4}{3}S_{1,0}S_{0,1}Q_{\mu}^{\frac{1}{3}}+2\mu A(S_{1,0}S_{0,1})Q_{\mu}.
\end{align*}
Using this fact, we compute
\begin{align*}
    &\text{Left-hand side of \eqref{estimateM2dou}}\\
    =&(\epsilon_2,L_{-,\mu}S_{1,0})-\frac{4}{3}b(Q_{\mu}^{\frac{1}{3}}S_{1,0}\epsilon_1,S_{1,0})-\frac{4}{3}d\cdot(Q_{\mu}^{\frac{1}{3}}S_{0,1}\epsilon_1,S_{1,0})-2\mu b(A(Q_{\mu}\epsilon_1)S_{1,0},S_{1,0})\\
    &-\mu d\cdot (A(Q_{\mu}\epsilon_1)S_{0,1},S_{1,0})+b(\epsilon_1,\Lambda S_{1,0})-d\cdot(\epsilon_1,\nabla S_{1,0})+2b(\epsilon_1,L_{+,\mu}T_{2,0})\\
    &+d\cdot(\epsilon_1,L_{+,\mu}T_{1,1})+\mathcal{O}(\mathcal{P}^2\|\epsilon\|_{L^2})\\
    =&(\epsilon_2,L_{-,\mu}S_{1,0})-\frac{4}{3}b(Q_{\mu}^{\frac{1}{3}}S_{1,0}\epsilon_1,S_{1,0})-\frac{4}{3}d\cdot(Q_{\mu}^{\frac{1}{3}}S_{0,1}\epsilon_1,S_{1,0})-2\mu b(A(Q_{\mu}\epsilon_1)S_{1,0},S_{1,0})\\
    &-\mu d\cdot (A(Q_{\mu}\epsilon_1)S_{0,1},S_{1,0})+b(\epsilon_1,\Lambda S_{1,0})-d\cdot(\epsilon_1,\nabla S_{1,0})\\
    &+2b\left(\epsilon_1,\frac{2}{3}Q_{\mu}^{\frac{1}{3}}S_{1,0}^2+S_{1,0}-\Lambda S_{1,0}+\mu A(S_{1,0}^2)Q_{\mu}\right)\\
    &+d\cdot\left(\epsilon_1,S_{0,1}-\Lambda S_{0,1}+\nabla S_{1,0}+\frac{4}{3}S_{1,0}S_{0,1}Q_{\mu}^{\frac{1}{3}}+2\mu A(S_{1,0}S_{0,1})Q_{\mu}\right)+\mathcal{O}(\mathcal{P}^2\|\epsilon\|_{L^2})\\
    =&(\epsilon_2,\Lambda Q_{\mu})-b(\epsilon_1,\Lambda S_{1,0})-d\cdot(\epsilon_1,\Lambda S_{0,1})+d\cdot(\epsilon_1,S_{0,1})+\mathcal{O}(\mathcal{P}^2\|\epsilon\|_{L^2})\\
    =&(\epsilon_2,\Lambda R_{\mathcal{P}})-(\epsilon_1,\Lambda R_{\mathcal{P}})+\mathcal{O}(\mathcal{P}^2\|\epsilon\|_{L^2}).
\end{align*}
Here we use the orthogonality conditions \eqref{orthogonality1dou}. This completes the proof of \eqref{estimateM2dou}.

By the similar argument as above, we can obtain \eqref{estimateM3dou}, \eqref{estimateM4dou} and \eqref{estimateM5dou}. Here we omit the details.

2. \textbf{The law for $b$.} We take the inner product of the equation \eqref{equ:mod1} of $\epsilon_1$ with $-\Lambda R_2$ and we sum it with the inner product of equation \eqref{equ:mod2} of $\epsilon_2$ with $\Lambda R_1$. We obtain after integrating by parts:
\begin{align*}
    &-(b_s+b^2)\big[(\partial_bR_1,-\Lambda R_2)+(\partial_bR_2,\Lambda R_1)\big]+(\partial_s\epsilon_1,-\Lambda R_2)+(\partial_s\epsilon_2,\Lambda\R_1)\\
    &+[(M_2(\epsilon)-b\Lambda\epsilon_1,\Lambda R_2)+(M_1(\epsilon)+b\Lambda\epsilon_2,\Lambda R_1)]+\left(\frac{\alpha_s}{\lambda}-d\right)\mathcal{O}(\mathcal{P})\\
    =&\left(\frac{\lambda_s}{\lambda}+b\right)\big[(\Lambda R_1+\Lambda\epsilon_1,-\Lambda R_2)+(\Lambda R_2+\Lambda\epsilon_2,\Lambda R_1)\big]\\
    &+\tilde{\gamma}_s\big[(R_2+\epsilon_2,-\Lambda R_2)+(R_1+\epsilon_1,\Lambda R_1)\big]+(\Im\Phi_b,-\Lambda R_2)-(\Re{\Phi}_b,\Lambda R_1)\\
    &+(P_2(\epsilon),\Lambda R_2)+(P_1(\epsilon),\Lambda R_1).
\end{align*}
From \eqref{estimateM1dou} and the orthogonality condition \eqref{orthogonality1dou}, we deduce that
\begin{align*}
   & -(b_s+b^2)((L_{-,\mu}S_{1,0},S_{1,0})+\mathcal{O}(b^2))+\left(\frac{\alpha_s}{\lambda}-d\right)\mathcal{O}(\mathcal{P})\\
    =&\Re{(\epsilon,R_{\mathcal{P}})}+(\Im\Phi_b,-\Lambda R_2)-(\Re{\Phi}_b,\Lambda R_1)+(P_2(\epsilon),\Lambda R_2)+(P_1(\epsilon),\Lambda R_1)\\
    &+\mathcal{O}((\mathcal{P}^{\frac{4}{3}}+|Mod(t)|)(\|\epsilon\|_{L^2}+\mathcal{P}^2)).
\end{align*}
Hence, by using the fact that $2\Re{(\epsilon,R_{\mathcal{P}})}=-\int|\epsilon|^2+\int(|u|^2-|Q_{\mu}|^2)+\mathcal{O}(b^4+|d|^2+|d\mathcal{P}^2|)$, we have
\begin{align*}
    &-(b_s+b^2)2e_{\mu}+\mathcal{O}(b^2))+\left(\frac{\alpha_s}{\lambda}-d\right)\mathcal{O}(\mathcal{P})\\
    =&-\frac{1}{2}\int|\epsilon|^2+(P_2(\epsilon),\Lambda R_2)+(P_1(\epsilon),\Lambda R_1)\\
    &+\mathcal{O}\left((\mathcal{P}^{\frac{4}{3}}+|Mod(t)|)(\|\epsilon\|_{L^2}+\mathcal{P}^2)+|\|u\|_{L^2}^2-\|Q_{\mu}\|_{L^2}|+b^4+|d|^2+|d\mathcal{P}^2|\right).
\end{align*}
3. \textbf{The law for $\lambda$.} We multiply both sides of the equation \eqref{equ:mod1} and \eqref{equ:mod2} by $-\partial_bR_2$ and $\partial_bR_1$, respectively. Adding this and using \eqref{estimateM2dou} yields, we can obtain
\begin{align*}
    &\left(\frac{\lambda_s}{\lambda}+b\right)(2e_{\mu}+\mathcal{O}(b^2))+(d_s+bd)\mathcal{O}(\mathcal{P})\\
    =&(R_2(\epsilon),\partial_bR_1)+(R_1(\epsilon),\partial_bR_2)+\mathcal{O}\left((\mathcal{P}^{\frac{4}{3}}+|Mod(t)|)(\|\epsilon\|_{L^2}+\mathcal{P}^2)+b^5+|d|^2\mathcal{P}\right).
\end{align*}
Here we also used that $(R_2,\partial_bR_2)+(R_1,\partial_bR_1)=b(S_{1,0},S_{1,0})+2b(Q_{\mu},T_{2,0})+d\cdot(T_{1,1},Q_{\mu})+\mathcal{O}(b^2)=\mathcal{O}(b^2)$, since $(S_{1,0},S_{1,0})+(Q_{\mu},T_{2,0})=0$ and $(T_{1,1},Q_{\mu})=0$.

4. \textbf{The law for $\tilde{\gamma}$.}  We multiply both sides of the equation \eqref{equ:mod1} and \eqref{equ:mod2} by $-\rho_2$ and $\rho_1$, respectively. Adding this and using \eqref{estimateM3dou} yields, we can obtain
\begin{align*}
    \tilde{\gamma}_s\big((Q_{\mu},\rho_1)+\mathcal{O}(b^2)\big)=-(b_s+b^2)\big((S_{1,0},\rho_1)+\mathcal{O}(b^2)\big)+\left(\frac{\lambda_s}{\lambda}+b\right)\mathcal{O}(b)\\
    +(P_2(\epsilon),\rho_2)+(P_1(\epsilon),\rho_1)+\mathcal{O}\left((b^{\frac{4}{3}}+|Mod(t)|)\|\epsilon\|_{L^2}+b^5\right).
\end{align*}
5. \textbf{The law for $d$.} We project \eqref{equ:mod1} and \eqref{equ:mod2} onto $-\nabla R_2$ and $\nabla R_1$, respectively. Then we have 
\begin{align*}
    &(d_s+bd)(-p_{\mu}+\mathcal{O}(\mathcal{P}^2))+(b_s+b^2)\mathcal{O}(\mathcal{P})\\
    =&(P_2,\nabla R_1)+(P_1,\nabla R_2)+\mathcal{O}((\mathcal{P}^2+|Mod(t)|)\|\epsilon\|_{L^2}+b^5+|d^2\mathcal{P}|).
\end{align*}
6. \textbf{The law for $\alpha$.} We project \eqref{equ:mod1} and \eqref{equ:mod2} onto $-\partial_dR_2$ and $\partial_dR_1$, respectively. Then we deduce that 
\begin{align*}
    &(b_s+b^2)\mathcal{O}(\mathcal{P})+\left(\frac{\alpha_s}{\lambda}-d\right)(p_{\mu}+\mathcal{O}(\mathcal{P}^2))\\
    =&(P_2,\partial_dR_1)+(P_1,\partial_dR_2)+\mathcal{O}((\mathcal{P}^2+|Mod(t)|)\|\epsilon\|_{L^2}+b^4+d^2+|d\mathcal{P}^2|).
\end{align*}

7. \textbf{Conclusion.} We collect the results in previous points  2,3,4,5,6 and estimate the nonlinear terms in $\epsilon$ by Sobolev inequalities. This gives us
\begin{align*}
    (A+B)Mod(t)=\mathcal{O}\Big((\mathcal{P}^{\frac{4}{3}}+|Mod(t)|)\|\epsilon\|_{L^2}+\|\epsilon\|_{L^2}^2+\|\epsilon\|_{H^1}^3\\
    +|\|u\|_{L^2}^2-\|Q_{\mu}\|_{L^2}^2|+b^4+d^2+|d\mathcal{P}^2|\Big).
\end{align*}
Here $A=\mathcal{O}(1)$ is invertible $9\times9$ matrix, whereas $B=\mathcal{O}(b)$ is some $9\times9$-matrix that is polynomial in $b$. Inverting this relation to compute $Mod(t)$ and computing the  Taylor expansion of $(A+B)^{-1}$ to sufficiently high order yields the desired result. This completes the proof of lemma \ref{lemma:modestimate}.
\end{proof}

\subsection{Refined energy identity}

In this subsection, our aim is to derive a general refined mixed/Morawetz type  estimate which allow us to derive a Lyapunov function for critical mass blow-up solutions.

Let $u\in H^1(\mathbb{R}^3)$ be a solution to \eqref{equ1:double} on $[t_0,0)$ and let $w\in H^1(\mathbb{R}^3)$ be an approximate solution to \eqref{equ1:double}:
\begin{align}\notag
    i\partial_tw+\Delta w+|w|^{\frac{4}{3}}w+\mu A(w^2)w=\psi,
\end{align}
with the a-priori bounds
\begin{align}\label{priori:estimate1dou}
    \|w\|_{L^2}\lesssim1,\,\,\|\nabla w\|_{L^2}\lesssim\lambda^{-1},\,\,\|w\|_{\dot{H}^{2}}\lesssim\lambda^{-2},
\end{align}
where $\delta>0$ is small enough. We then decompose $u=w+\tilde{u}$ so that $\tilde{u}\in H^1(\mathbb{R}^3)$ satisfies:
\begin{align}\label{equ:approxiamteE2}
    i\partial_t\tilde{u}+\Delta\tilde{u}+\left(|u|^{\frac{4}{3}}u-|w|^{\frac{4}{3}}w\right)+\mu \left(A(u^2)u-A(w^2)w\right)=-\psi,
\end{align}
and assume the a-priori bounds on $\tilde{u}$:
\begin{align}\label{priori:estimate2dou}
    \|\tilde{u}\|_{L^2}\lesssim\lambda^{2},\,\,\,\|\nabla\tilde{u}\|_{L^2}\lesssim\lambda,
\end{align}
and on the geometrical parameters:
\begin{align}\label{priori:estimate3dou}
    |\lambda\lambda_t+b|\lesssim\lambda^4,\,\,b\sim\lambda,\,\,|b_t|\lesssim1,\,\,|\alpha_t|\lesssim\lambda^2
\end{align}
for some non - negative parameters $0<\lambda,b\ll1$.

Let $M>0$ be a large enough constant, which will be chosen later, and let $\phi:\mathbb{R}^3\rightarrow\mathbb{R}$ be a smooth radially symmetric cutoff function with
\begin{align}\notag
    \phi^\prime(r)=\begin{cases}
    r\,\,\,&\text{for}\,\,\, r\leq1,\\
    3-e^{-r}\,\,\, &\text{for}\,\,\, r\geq2,
    \end{cases}
\end{align}
and the convexity condition
\begin{align}\notag
	\phi^{\prime\prime}(r)\geq0 \,\,\text{for}\,\,r\geq0.
\end{align}
Let
\begin{align*}
    &F(u)=\frac{3}{10}|u|^{\frac{10}{3}},\,\,\,f(u)=|u|^{\frac{4}{3}}u,\,\,\,F^\prime(u)\cdot h=\Re{(f(u)\bar{h})};\\
    &G(u)=\frac{1}{4}A(u^2)|u|^2,\,\,\,g(u)=A(u^2)u\,\,\,
    G^\prime(u)\cdot h=\Re{(g(u)\bar{h})}.
\end{align*}
Now we give the following generalized energy estimate.
\begin{lemma}\label{lemma:energyestimatedou} \textbf{(Generalized energy estimate).} Let
\begin{align*}
    J=&\frac{1}{2}\int|\nabla\tilde{u}|^2+\frac{1}{2}\int\frac{|\tilde{u}|^2}{\lambda^2}-\int[F(w+\tilde{u})-F(w)-F^{\prime}(w)\cdot\tilde{u}]\notag\\
    &-\mu\int[G(w+\tilde{u})-G(w)-G^{\prime}(w)\cdot\tilde{u}]+\frac{1}{2}\frac{b}{\lambda}\Im\left(\int M\nabla\phi\left(\frac{x-\alpha}{M\lambda}\right)\cdot\nabla\tilde{u}\tilde{u}\right).
\end{align*}
Then the following holds:
\begin{align}\label{refine:energyestiamte}
    \frac{dJ}{dt}=&\Im{\left(\Delta\psi-\frac{1}{\lambda^2}\psi+f^{\prime}(w)\cdot\psi+\mu g^{\prime}(w)\cdot\psi,\bar{\tilde{u}}\right)}+\frac{b}{\lambda^4}\int|\tilde{u}|^2\notag\\
    &-\frac{1}{\lambda^2}\Im\int\left(\frac{2}{3}|w|^{-\frac{2}{3}}w^2\bar{\tilde{u}}^2+\mu(A(w^2)\tilde{u}+A(2\Re(\bar{w}\tilde{u}))\bar{\tilde{u}}
   \right)\notag\\
   &-\Re{\Big(\partial_tw,\overline{(f(\tilde{u}+w)-f(w)-f^{\prime}(w)\cdot\tilde{u}-f(\tilde{u})}\Big)}\notag\\
   &-\Re{(\partial_tw,A(\tilde{u}^2)w+A(2\Re(w\bar{\tilde{u}})\tilde{u})}\notag\\
   &+\frac{b}{\lambda^2}\Re\left(\int\nabla^2\phi\left(\frac{x-\alpha}{M\lambda}\right)(\nabla\tilde{u},\overline{\nabla\tilde{u}})\right)-\frac{1}{4}\frac{b}{M^2\lambda^4}\left(\int\Delta^2\phi\left(\frac{x-\alpha}{M\lambda}\right)|\tilde{u}|^2\right)\notag\\
   &+\frac{b}{\lambda}\Re\left(\int M\nabla\phi\left(\frac{x-\alpha}{M\lambda}\right)\left(\frac{10}{9}|w|^{-\frac{2}{3}}w|\tilde{u}|^2-\frac{1}{9}|w|^{-\frac{8}{3}}\bar{w}^3\tilde{u}^2+\frac{5}{9}|w|^{-\frac{2}{3}}\bar{w}\tilde{u}^2\right)\cdot\overline{\nabla w}\right)\notag\\
    &+\frac{\mu b}{\lambda}\Re\left(\int M\nabla\phi\left(\frac{x-\alpha}{M\lambda}\right)\left(A(w\nabla\bar{w}+\bar{w}\nabla w)\tilde{u}+A(u\nabla\bar{w}+\bar{\tilde{u}}\nabla w)w+A(2\Re(w\bar{\tilde{u}}))\cdot\nabla w\right)\right)\notag\\
    &+\Im\left(\int\left[i\frac{b}{\lambda}M\nabla\phi\left(\frac{x-\alpha}{M\lambda}\right)\cdot\nabla\psi+i\frac{b}{2\lambda^2}\Delta\phi\left(\frac{x-\alpha}{M\lambda}\right)\psi\right]\bar{\tilde{u}}\right)\notag\\
   &+\mathcal{O}\left(\lambda^2\|\psi\|_{L^2}+\|\tilde{u}\|_{H^1}^{\frac{2}{3}}\right).
\end{align}
\end{lemma}
\begin{proof}
\textbf{Step 1.} Algebraic derivation of the energy part.
Using \eqref{equ:approxiamteE2}, a computation shows that
\begin{align}\label{refine:energy11dou}
    &\frac{d}{dt}\Bigg\{\frac{1}{2}\int|\nabla\tilde{u}|^2+\frac{1}{2}\int\frac{|\tilde{u}|^2}{\lambda^2}-\int[F(w+\tilde{u})-F(w)-F^{\prime}(w)\cdot\tilde{u}]\notag\\
    &-\mu\int[G(w+\tilde{u})-G(w)-G^{\prime}\cdot\tilde{u}]\Bigg\}\notag\\
    =&-\Re{\left(\partial_t\tilde{u},\overline{\Delta\tilde{u}-\frac{1}{\lambda^2}\tilde{u}+(f(u)-f(w))+\mu(g(u)-g(w))}\right)}-\frac{\lambda_t}{\lambda^3}\int|\tilde{u}|^2\notag\\
    &-\Re{\Big(\partial_tw,\overline{(f(\tilde{u}+w)-f(w)-f^{\prime}(w)\cdot\tilde{u}+\mu(g(\tilde{u}+w)-g(w)-g^{\prime}(w)\cdot\tilde{u})}\Big)}\notag\\
    =&\Im{\left(\psi,\overline{\Delta\tilde{u}-\frac{1}{\lambda^2}\tilde{u}+(f(u)-f(w))+\mu(g(u)-g(w))}\right)}-\frac{\lambda_t}{\lambda^3}\int|\tilde{u}|^2\notag\\
    &-\Re{\Big(\partial_tw,\overline{(f(\tilde{u}+w)-f(w)-f^{\prime}(w)\cdot\tilde{u}+\mu(g(\tilde{u}+w)-g(w)-g^{\prime}(w)\cdot\tilde{u})}\Big)}\notag\\
    &-\frac{1}{\lambda^2}\Im{\Big(f(u)-f(w)+\mu(g(u)-g(w)),\bar{\tilde{u}}\Big)}\notag\\
    =&\Im{\left(\psi,\overline{\Delta\tilde{u}-\frac{1}{\lambda^2}\tilde{u}+f^{\prime}(w)\cdot\tilde{u}+\mu g^{\prime}(w)\cdot\tilde{u}}\right)}-\frac{\lambda_t}{\lambda^3}\int|\tilde{u}|^2\notag\\
    &-\frac{1}{\lambda^2}\Im{\Big(f(u)-f(w)+\mu(g(u)-g(w)),\bar{\tilde{u}}\Big)}\notag\\
    &+\Im{\left(\psi-\frac{1}{\lambda^2}\tilde{u},\overline{f(u)-f(w)-f^{\prime}(w)\cdot\tilde{u}-\mu(g(u)-g(w)-g^{\prime}(w)\cdot\tilde{u})}\right)}\notag\\
    &-\Re{\Big(\partial_tw,\overline{(f(\tilde{u}+w)-f(w)-f^{\prime}(w)\cdot\tilde{u}+\mu(g(\tilde{u}+w)-g(w)-g^{\prime}(w)\cdot\tilde{u})}\Big)},
\end{align}
where we used that
\begin{align*}
    &f^{\prime}(w)\cdot\tilde{u}=\frac{5}{3}|w|^{\frac{4}{3}}\tilde{u}+\frac{2}{3}|w|^{-\frac{2}{3}}w^2\bar{\tilde{u}},~~g^{\prime}(w)\cdot\tilde{u}=A(w^2)\tilde{u}+A(2\Re(\bar{w}\tilde{u}))w.
\end{align*}
From \eqref{priori:estimate3dou}, we obtain that
\begin{align}\label{refine:energy1dou}
    -\frac{\lambda_t}{\lambda^3}\int|\tilde{u}|^2=\frac{b}{\lambda^4}\int|\tilde{u}|^2-\frac{1}{\lambda^4}(\lambda\lambda_t+b)\|\tilde{u}\|_{L^2}^2=\frac{b}{\lambda^4}\int|\tilde{u}|^2+\mathcal{O}(\|\tilde{u}\|_{H^1}^2).
\end{align}
Next, we estimate
\begin{align}\label{refine:energy2dou}
    &\left|\Im{\left(\psi-\frac{1}{\lambda^2}\tilde{u},\overline{f(u)-f(w)-f^{\prime}(w)\cdot\tilde{u}}\right)}\right|\notag\\
    \lesssim&(\|\psi\|_{L^2}+\lambda^{-2}\|\tilde{u}\|_{L^2})\big(\|f(u)-f(w)-f^{\prime}(w)\cdot\tilde{u}\|_{L^2}\big)\notag\\
    \lesssim&(\|\psi\|_{L^2}+\lambda^{-2}\|\tilde{u}\|_{L^2})(\||\tilde{u}|^{\frac{7}{3}}+|w|^{\frac{1}{3}}|\tilde{u}|^2\|_{L^2})\notag\\
    \lesssim&(\|\psi\|_{L^2}+\lambda^{-2}\|\tilde{u}\|_{L^2})\left[\|\nabla\tilde{u}\|_{L^2}^{2}\|\tilde{u}\|_{L^2}^{\frac{1}{3}}+\|w\|_{L^2}^{\frac{1}{3}}\|u\|_{L^6}^{2}\right]\notag\\
    \lesssim&\lambda^2\|\psi\|_{L^2}+\mathcal{O}(\|\tilde{u}\|_{H^1}^2).
\end{align}
Here we also used the following inequality
$|h(u+v)-h(u)-h^{\prime}(u)\cdot v|\lesssim|v|^{p}+|u|^{p-2}|v|^2,$
 for $p>2$, $h(u)=|u|^{p-1}u$.
On the other hand,  we have
\begin{align}\label{refine:energy22dou}
    &\left|\Im{\left(\psi-\frac{1}{\lambda^2}\tilde{u},\overline{g(u)-g(w)-g^{\prime}(w)\cdot\tilde{u}}\right)}\right|\notag\\
    =&\left|\Im{\left(\psi-\frac{1}{\lambda^2}\tilde{u},\overline{A(\tilde{u}^2)(w+\tilde{u})+A(w\bar{\tilde{u}}+\bar{w}\tilde{u})\tilde{u}}\right)}\right|\notag\\
    \lesssim&\left(\|\psi\|_{L^2}+\lambda^{-2}\|\tilde{u}\|_{L^2}\right)\left(\|\tilde{u}\|_{\dot{H}^1}^2(\|w\|_{L^2}+\|\tilde{u}\|_{L^2})+\|A(w\tilde{u})\|_{L^4}\|u\|_{L^4}\right)\notag\\
    \lesssim&\left(\|\psi\|_{L^2}+\lambda^{-2}\|\tilde{u}\|_{L^2}\right)\left(\|\tilde{u}\|_{\dot{H}^1}^2(1+\lambda^2)+\|u\|_{L^{12/5}}|u\|_{L^4}\right)\notag\\
    \lesssim&\lambda^2\|\psi\|_{L^2}+\|\tilde{u}\|_{{H}^1}^2.
\end{align}
 Here we use the Hardy inequality,
H\"older inequality and the priori estimate \eqref{priori:estimate2dou}.

For the term that contain the $\partial_tw$, we replace $\partial_tw$ using \eqref{equ:approximate:double}, integrate by parts and then rely on \eqref{priori:estimate1dou} to estimate
\begin{align}\label{refine:energy3dou}
    \left|\int \partial_tw |\tilde{u}|^{\frac{7}{3}}\right|\lesssim&\left(\|w\|_{\dot{H}^2}+\|w^{\frac{7}{3}}\|_{L^2}+\mu \|A(|w|^2)|w|\|_{L^2}+\|\psi\|_{L^2}\right)\|\tilde{u}^{\frac{7}{3}}\|_{L^2}\notag\\
    \lesssim&\lambda^{-2}\|\tilde{u}\|_{L^{\frac{14}{3}}}^{\frac{7}{3}}+\|\psi\|_{L^2}\|\tilde{u}\|_{L^{\frac{14}{3}}}^{\frac{7}{3}}\notag\\
    \lesssim &\|\tilde{u}\|_{H^1}^{\frac{2}{3}}+\lambda^2\|\psi\|_{L^2}.
\end{align}
Here we also used the H\"{o}lder inequality and Hardy inequality. And
\begin{align}\label{refine:energy4dou}
    \left|\int\partial_tw|\overline{A(|\tilde{u}|^2)|\tilde{u}|}\right|\lesssim&\|w\|_{\dot{H}^2}\|A(|\tilde{u}|^2)|\tilde{u}|\|_{L^2}+\|w\|_{L^{\frac{14}{3}}}^{\frac{7}{3}}\|A(|\tilde{u}|^2)|\tilde{u}|\|_{L^2}\notag\\
    &+\mu\|A(|w|^2)|w|\|_{L^2}\|A(|\tilde{u}|^2)|\tilde{u}|\|_{L^2}+\|\psi\|_{L^2}\|A(|\tilde{u}|^2)|\tilde{u}|\|_{L^2}\notag\\
    \lesssim&\frac{1}{\lambda^{2}}\|\tilde{u}\|_{\dot{H}^1}^2\|\tilde{u}\|_{L^2}+\|w\|_{\dot{H}^1}^2\|w\|_{L^2}\|\tilde{u}\|_{\dot{H}^1}^2\|\tilde{u}\|_{L^2}\notag\\
    &+\mu\|w\|_{L^{\frac{14}{3}}}^{\frac{7}{3}}\|w\|_{L^4}^{\frac{4}{3}}\|\tilde{u}\|_{\dot{H}^1}^2\|\tilde{u}\|_{L^2}+\|\psi\|_{L^2}\|\tilde{u}\|_{\dot{H}^1}^2\|\tilde{u}\|_{L^2}\notag\\
    \lesssim&\lambda^2\|\psi\|_{L^2}+\mathcal{O}\left(\|\tilde{u}\|_{H^1}^2\right).
\end{align}
We now insert \eqref{refine:energy1dou}, \eqref{refine:energy2dou}, \eqref{refine:energy22dou},\eqref{refine:energy3dou} and \eqref{refine:energy4dou} into \eqref{refine:energy11dou}, we have
\begin{align*}
    &\frac{d}{dt}\Bigg\{\frac{1}{2}\int|\nabla\tilde{u}|^2+\frac{1}{2}\int\frac{|\tilde{u}|^2}{\lambda^2}-\int[F(w+\tilde{u})-F(w)-F^{\prime}(w)\cdot\tilde{u}]-\mu\int[G(w+\tilde{u})-G(w)-G^{\prime}\cdot\tilde{u}]\Bigg\}\notag\\
    =&\Im{\left(\Delta\psi-\frac{1}{\lambda^2}\psi+f^{\prime}(w)\cdot\psi+\mu g^{\prime}(w)\cdot\psi,\bar{\tilde{u}}\right)}+\frac{b}{\lambda^4}\int|\tilde{u}|^2\notag\\
    &-\frac{1}{\lambda^2}\Im\int\left(\frac{2}{3}|w|^{-\frac{2}{3}}w^2\bar{\tilde{u}}^2+\mu(A(w^2)\tilde{u}+2A(\Re\bar{w}\tilde{u})\bar{\tilde{u}}
   \right)-\Re{\Big(\partial_tw,\overline{(f(\tilde{u}+w)-f(w)-f^{\prime}(w)\cdot\tilde{u}-f(\tilde{u})}\Big)}\notag\\
   &-\Re{(\partial_tw,A(\tilde{u}^2)w+A(w\bar{\tilde{u}}+\bar{w}\tilde{u})\tilde{u})}+\mathcal{O}\left(\lambda^2\|\psi\|_{L^2}+\|\tilde{u}\|_{H^1}^{\frac{2}{3}}\right).
\end{align*}
\textbf{Step 2.} Algebraic derivation of the localized virial part. Let
\begin{align*}
    \nabla\tilde{\phi}(t,x)=\frac{b}{\lambda}M\nabla\phi\left(\frac{x-\alpha}{M\lambda}\right).
\end{align*}
Then
\begin{align}\label{virial1dou}
    &\frac{1}{2}\frac{d}{dt}\left( \frac{b}{\lambda}\Im\left(\int M\nabla\phi\left(\frac{x-\alpha}{M\lambda}\right)\cdot\nabla\tilde{u} \bar{\tilde{u}}\right)\right)\notag\\
    =&\frac{1}{2}\Im\left(\int\partial_t\nabla\tilde{\phi}\cdot\nabla\tilde{u}\bar{\tilde{u}}\right)+\Re{\left(\int i\partial_t\tilde{u}\overline{\left(\frac{1}{2}\Delta\tilde{u}+\nabla\tilde{\phi}\cdot\nabla\tilde{u}\right)}\right)}.
\end{align}
Using \eqref{priori:estimate3dou}, we estimate
\begin{align*}
    |\partial_t\nabla\tilde{\phi}|\lesssim\frac{1}{\lambda^3}\left(|\lambda^2b_t|+|\lambda\lambda_t+b|\right)+\frac{b}{\lambda}|\alpha_t|\lesssim\frac{1}{\lambda},
\end{align*}
from which
\begin{align}\label{virial:est1dou}
    \left|\Im\left(\int\partial_t\nabla\tilde{\phi}\cdot\nabla\tilde{u}\bar{\tilde{u}}\right)\right|\lesssim\frac{1}{\lambda}\|\tilde{u}\|_{L^2}\|\nabla\tilde{u}\|_{L^2}=\mathcal{O}\left(\frac{1}{\lambda^2}\|\tilde{u}\|_{L^2}^2+\|\tilde{u}\|_{H^1}^2\right).
\end{align}
The second term in \eqref{virial1dou} corresponds to the localized Morawetz estimate, and from \eqref{equ:approxiamteE2} and integration by parts, we get
\begin{align}\label{virial2dou}
   &\Re{\left(\int i\partial_t\tilde{u}\overline{(\frac{1}{2}\Delta\tilde{u}+\nabla\tilde{\phi}\cdot\nabla\tilde{u})}\right)}\notag\\
    =&\frac{b}{\lambda^2}\Re\left(\int\nabla^2\phi\left(\frac{x-\alpha}{M\lambda}\right)(\nabla\tilde{u},\bar{\nabla\tilde{u}})\right)-\frac{1}{4}\frac{b}{M^2\lambda^4}\left(\int\Delta^2\phi\left(\frac{x-\alpha}{M\lambda}\right)|\tilde{u}|^2\right)\notag\\
    &-\frac{b}{\lambda}\Re\left(\int M\nabla\phi\left(\frac{x-\alpha}{M\lambda}\right)\left(\left(|u|^{\frac{4}{3}}u-|w|^{\frac{4}{3}}w\right)+\mu\left(A(|u|^2)u-A(w^2)w\right)\right)\cdot\overline{\nabla\tilde{u}}\right)\notag\\
    &-\frac{1}{2}\frac{b}{\lambda^2}\Re\left(\int \Delta\phi\left(\frac{x-\alpha}{M\lambda}\right)\left(\left(|u|^{\frac{4}{3}}u-|w|^{\frac{4}{3}}w\right)+\mu\left(A(u^2)u-A(w^2)w\right)\right)\bar{\tilde{u}}\right)\notag\\
    &-\frac{b}{\lambda}\Re\left(\int M\nabla\phi\left(\frac{x-\alpha}{M\lambda}\right)\psi\cdot\bar{\nabla\tilde{u}}\right)-\frac{1}{2}\frac{b}{\lambda^2}\Re\left(\int\Delta\phi\left(\frac{x-\alpha}{M\lambda}\right)\psi\bar{\tilde{u}}\right).
\end{align}
We now estimate the nonlinear terms
\begin{align}\label{virial:est2dou}
    &\Bigg|-\frac{b}{\lambda}\Re\left(\int M\nabla\phi\left(\frac{x-\alpha}{M\lambda}\right)\left((f(u)-f(w)-f^{\prime}(w)\cdot\tilde{u})+\mu(g(u)-g(w)-g\prime(w)\cdot\tilde{u})\right)\cdot\overline{\nabla\tilde{u}}\right)\notag\\
    &-\frac{1}{2}\frac{b}{\lambda^2}\Re\left(\int\Delta\phi\left(\frac{x-\alpha}{M\lambda}\right)(f(u)-f(w)-f^{\prime}(w)\cdot\tilde{u})\bar{\tilde{u}}+\mu(g(u)-g(w)-g\prime(w)\cdot\tilde{u})\bar{\tilde{u}}\right)\Bigg|\notag\\
    \lesssim&\frac{b}{\lambda}\Re\int M\nabla\phi\left(\frac{x-\alpha}{M\lambda}\right)\left(|\tilde{u}|^{\frac{4}{3}+1}+|w|^{\frac{1}{3}}|\tilde{u}|^2+\mu(A(\tilde{u}^2)(w+\tilde{u})+A(w\bar{\tilde{u}}+\bar{w}\tilde{u})\tilde{u})\right)\cdot\overline{\nabla\tilde{u}}\notag\\
    &+\frac{1}{2}\frac{b}{\lambda^2}\Re\left(\int\Delta\phi\left(\frac{x-\alpha}{M\lambda}\right)\left(|\tilde{u}|^{\frac{4}{3}+1}+|w|^{\frac{1}{3}}|\tilde{u}|^2+\mu(A(\tilde{u}^2)(w+\tilde{u})+A(w\bar{\tilde{u}}+\bar{w}\tilde{u})\tilde{u})\right)\bar{\tilde{u}}\right)\notag\\
    \lesssim&\left(\|\tilde{u}\|_{L^{\frac{14}{3}}}^{\frac{7}{3}}+\|w\|_{L^2}^{\frac{1}{3}}\|\tilde{u}\|_{L^6}^2+\mu\|\tilde{u}\|_{\dot{H}^1}^2(\|w\|_{L^2}+\|\tilde{u}\|_{L^2})+\|w\|_{L^2}\|\tilde{u}\|_{L^{12/5}}\|\tilde{u}\|_{L^4} \right)\|\nabla\tilde{u}\|_{L^2}\notag\\
    &+\frac{1}{\lambda}\left(\|\tilde{u}\|_{L^{\frac{10}{3}}}^{\frac{10}{3}}+\|w\|_{L^2}^{\frac{1}{3}}\|\tilde{u}\|_{L^{\frac{18}{5}}}^3+\mu\|\tilde{u}\|_{\dot{H}^1}^2(\|w\|_{L^2}+\|\tilde{u}\|_{L^2})+\|w\|_{L^2}\|\tilde{u}\|_{L^{12/5}}\|\tilde{u}\|_{L^4}\|\tilde{u}\|_{L^2}\right)\notag\\
    \lesssim&\|\tilde{u}\|_{H^1}^2.
\end{align}
The remaining terms in \eqref{virial2dou} are integrated by parts:
\begin{align}\label{virial:est3dou}
  &-\frac{b}{\lambda}\Re\left(\int M\nabla\phi\left(\frac{x-\alpha}{M\lambda}\right)\psi\cdot\bar{\nabla\tilde{u}}\right)-\frac{1}{2}\frac{b}{\lambda^2}\Re\left(\int\Delta\phi\left(\frac{x-\alpha}{M\lambda}\right)\psi\bar{\tilde{u}}\right)\notag\\
  =&\Im\left(\int\left[i\frac{b}{\lambda}M\nabla\phi\left(\frac{x-\alpha}{M\lambda}\right)\cdot\nabla\psi+i\frac{b}{2\lambda^2}\Delta\phi(\frac{x-\alpha}{M\lambda})\psi\right]\bar{\tilde{u}}\right).
\end{align}
For the local term, we have 
\begin{align}\label{virial:est4dou}
    &-\frac{b}{\lambda}\Re\left(\int M\nabla\phi\left(\frac{x-\alpha}{M\lambda}\right)(f^{\prime}(w)\cdot\tilde{u})\cdot\overline{\nabla\tilde{u}}\right)-\frac{1}{2}\frac{b}{\lambda^2}\Re\left(\int \Delta\phi\left(\frac{x-\alpha}{M\lambda}\right)(f^{\prime}(w)\cdot\tilde{u})\bar{\tilde{u}}\right)\notag\\
    =&\frac{b}{\lambda}\Re\left(\int M\nabla\phi\left(\frac{x-\alpha}{M\lambda}\right)\left(\frac{10}{9}|w|^{-\frac{2}{3}}w|\tilde{u}|^2-\frac{1}{9}|w|^{-\frac{8}{3}}\bar{w}^3\tilde{u}^2+\frac{5}{9}|w|^{-\frac{2}{3}}\bar{w}\tilde{u}^2\right)\cdot\overline{\nabla w}\right).
\end{align}
For the non-local term, we have 
\begin{align}\label{virial:est5dou}
    &-\frac{b}{\lambda}\Re\left(\int M\nabla\phi\left(\frac{x-\alpha}{M\lambda}\right)(g^{\prime}(w)\cdot\tilde{u})\cdot\overline{\nabla\tilde{u}}\right)-\frac{1}{2}\frac{b}{\lambda^2}\Re\left(\int \Delta\phi\left(\frac{x-\alpha}{M\lambda}\right)(g^{\prime}(w)\cdot\tilde{u})\bar{\tilde{u}}\right)\notag\\
    =&\frac{b}{\lambda}\Re\left(\int M\nabla\phi\left(\frac{x-\alpha}{M\lambda}\right)\left(2A(\Re w\nabla\bar{w})\tilde{u}+2A(\Re\tilde{u}\nabla\bar{w})w+2A(\Re w\bar{\tilde{u}})\cdot\nabla w\right)\right).
\end{align}
Injecting \eqref{virial:est1dou}, \eqref{virial:est2dou}, \eqref{virial:est3dou}, \eqref{virial:est4dou},\eqref{virial:est5dou} into \eqref{virial2dou} yields after a further integration by parts
\begin{align*}
     &\Re{\left(\int i\partial_t\tilde{u}\overline{\left(\frac{1}{2}\Delta\tilde{u}+\nabla\tilde{\phi}\cdot\nabla\tilde{u}\right)}\right)}\\
      =&\frac{b}{\lambda^2}\Re\left(\int\nabla^2\phi\left(\frac{x-\alpha}{M\lambda}\right)(\nabla\tilde{u},\overline{\nabla\tilde{u}})\right)-\frac{1}{4}\frac{b}{M^2\lambda^4}\left(\int\Delta^2\phi\left(\frac{x-\alpha}{M\lambda}\right)|\tilde{u}|^2\right)\\
      &+\frac{b}{\lambda}\Re\left(\int M\nabla\phi\left(\frac{x-\alpha}{M\lambda}\right)\left(\frac{10}{9}|w|^{-\frac{2}{3}}w|\tilde{u}|^2-\frac{1}{9}|w|^{-\frac{8}{3}}\bar{w}^3\tilde{u}^2+\frac{5}{9}|w|^{-\frac{2}{3}}\bar{w}\tilde{u}^2\right)\cdot\overline{\nabla w}\right)\\
      &+\frac{\mu b}{\lambda}\Re\left(\int M\nabla\phi\left(\frac{x-\alpha}{M\lambda}\right)\left(2A(\Re w\nabla\bar{w})\tilde{u}+2A(\Re\tilde{u}\nabla\bar{w})w+2A(\Re w\bar{\tilde{u}})\cdot\nabla w\right)\right)\\
      &+\Im\left(\int\left[i\frac{b}{\lambda}M\nabla\phi\left(\frac{x-\alpha}{M\lambda}\right)\cdot\nabla\psi+i\frac{b}{2\lambda^2}\Delta\phi\left(\frac{x-\alpha}{M\lambda}\right)\psi\right]\bar{\tilde{u}}\right)\\
      &+\mathcal{O}\left(\|\tilde{u}\|_{H^1}^2\right).
\end{align*}
This completes the proof of lemma \ref{lemma:energyestimatedou}.
\end{proof}

\subsection{Backwards propagation of smallness}
In this subsection, we first application of the energy estimate \eqref{refine:energyestiamte} is a bootstrap control on critical mass solution to \eqref{equ1:double}. More precisely, let $u\in H^1(\mathbb{R}^3)$ be a solution to \eqref{equ1:double} defined on $[t_0,0)$. Let $t_0<t_1<0$ and assume that $u$ admits on $[t_0,t_1]$ a geometrical decomposition of the form:
\begin{align*}
    u(t,x)=\frac{1}{\lambda^{\frac{3}{2}}(t)}[R_{\mathcal{P}}+\epsilon]\left(t,\frac{x-\alpha(t)}{\lambda(t)}\right)e^{i\gamma(t)},
\end{align*}
where $\epsilon=\epsilon_1+i\epsilon_2\in H^1(\mathbb{R}^3)$ satisfies the orthogonality conditions \eqref{orthogonality1dou} and $\|\epsilon(t)\|_{H^1}+|b(t)|+|d(t)|\ll1$. Let
\begin{align}\notag
    \tilde{u}(t,x)=\frac{1}{\lambda^{\frac{3}{2}}(t)}\epsilon\left(t,\frac{x-\alpha(t)}{\lambda(t)}\right)e^{i\gamma(t)}.
\end{align}
Assume that the energy $E_{0,\mu}$ satisfies the $E_{0,\mu}=E_{\mu}(u)>0$ and define the constant
\begin{align}\label{backBdefine}
    B_\mu=\sqrt{\frac{e_{\mu}}{E_{0,\mu}}},
\end{align}
with the constant $e_\mu=\frac{1}{2}(L_{-,\mu}S_{1,0},S_{1,0})>0$. Moreover, Let $P_{0,\mu}=P_{\mu}(u_0)$ be the linear momentum and define the vector
\begin{align}\label{backDdefine}
    D_{\mu}=\frac{P_{0,\mu}}{p_{\mu}},
\end{align}
with the universal constant $p_{\mu}=2(L_{-,\mu}S_{0,1},S_{0,1})$. We claim the following backwards propagation estimates:
\begin{lemma}\label{lemmabackwardsdou}
\textbf{(Backwards propagation of smallness).} Assume that there holds for some $t_1<0$ close enough to $0$:
\begin{align*}
    &\left|\|u\|_{L^2}^2-\|Q_{\mu}\|_{L^2}^2\right|\lesssim\lambda^4(t_1),~~\|\nabla\tilde{u}(t_1)\|_{L^2}^2+\frac{\|\tilde{u}(t_1)\|_{L^2}^2}{\lambda^2(t_1)}\lesssim\lambda^2(t_1),\\
   & \left|\lambda(t_1)+\frac{t_1}{B_\mu}\right|\lesssim\lambda^2(t_1),\,\,\left|\frac{b(t_1)}{\lambda(t_1)}-\frac{1}{B_\mu}\right|\lesssim\lambda^2(t_1),\,\,\left|\frac{d(t_1)}{\lambda^2(t_1)}-D_{\mu}\right|\lesssim\lambda^2(t_1).
\end{align*}
Then there exists a backwards time $t_0$ depending only on $B_\mu$ such that for any $t\in[t_0,t_1]$,
\begin{align*}
    &\|\nabla\tilde{u}(t)\|_{L^2}^2+\frac{\|\tilde{u}(t)\|_{L^2}^2}{\lambda^2(t)}\lesssim\|\nabla\tilde{u}(t_1)\|_{L^2}^2+\frac{\|\tilde{u}(t_1)\|_{L^2}^2}{\lambda^2(t_1)}+\lambda^{6}(t),\\\notag
    &\left|\frac{b}{\lambda}(t)-\frac{1}{B_\mu}\right|\lesssim\lambda^2(t),\,\,
    \left|\lambda(t)+\frac{t}{B_\mu}\right|\lesssim\lambda^2(t),\,\,\left|\frac{d(t)}{\lambda^2(t)}-D_{\mu}\right|\lesssim\lambda^2(t).
\end{align*}
\end{lemma}
\begin{proof}
By the similar argument as \cite{RS2011JAMS,GL2021CPDE,GL2021JFA,KLR2013ARMA}, we can obtain this Lemma \ref{lemmabackwardsdou}. Here we omit the details.
\end{proof}

\section{Existence of critical mass blow-up solutions}
In this section, we prove the following result, which in particular yields Theorem \ref{theorem:minimialD}. 
\begin{lemma}
\textbf{(Existence of critical mass blow-up solution).} Let $\gamma_0\in\mathbb{R}$, $x_0\in\mathcal{R}^3$, $B_\mu$ and $D_{\mu}$ be given by \eqref{backBdefine} and \eqref{backDdefine}, respectively. Then there exist $t_0<0$ and a solution $u_c\in \mathcal{C}([t_0,0),H^1(\mathbb{R}^3))$ to \eqref{equ1:double} which blows up at $T=0$ with \begin{align*}
    E_{\mu}(u_c)=E_{0,\mu}(u_0),\,\,P_{\mu}(u_c)=P_{0,\mu}(u_0)\,\,\text{and}\,\,\|u_c\|_{L^2}=\|Q_{\mu}\|_{L^2}.
\end{align*}
Furthermore, the solution admits on $[t_0,0)$ a geometrical decomposition:
\begin{align}\notag
    u_c(t,x)=\frac{1}{\lambda_c^{\frac{3}{2}}(t)}[R_{\mathcal{P}_c}+\epsilon_c]\left(t,\frac{x-\alpha_c(t)}{\lambda_c(t)}\right)e^{i\gamma_c(t)}=\tilde{R}_{\mathcal{P}_c}+\tilde{u}_c,
\end{align}
where $\epsilon_c$ satisfies the orthogonality conditions \eqref{orthogonality1dou} and the following bounds hold:
\begin{align*}
    &\|\tilde{u}_c\|_{L^2}^2\lesssim\lambda_c^4,\,\,\|\tilde{u}_c\|_{H^1}^2\lesssim\lambda^{2}_c,\\
    &\lambda_c+\frac{t}{B_\mu}=\mathcal{O}(\lambda_c^3),\,\,\frac{b_c}{\lambda_c}-\frac{1}{B_\mu}=\mathcal{O}(\lambda_c^2),\,\,\frac{d_c}{\lambda^2_c}-\frac{1}{D_{\mu}}=\mathcal{O}(\lambda_c^2)\,\,\gamma_c=-\frac{B_\mu^2}{t}+\gamma_0+\mathcal{O}(\lambda_c).
\end{align*}
\end{lemma}
\begin{proof}
By the similar argument as
\cite{RS2011JAMS,MRS2013Amer,M1990CMP,MRS2014Duke,LMR2016RMI}

\cite{KMR2009CPAM}. We can obtain this result. Here we only show the uniform $\dot{H}^{\frac{3}{2}}(\mathbb{R}^3)$ bound;
\begin{align}\label{minimal5:2dou}
    \|\tilde{u}_n\|_{L^{\infty}\left([t,t_n],\dot{H}^{\frac{3}{2}}(\mathbb{R}^3)\right)}\lesssim\lambda_n^{\frac{1}{2}}(t).
\end{align}
Indeed, our point is again the identity
\begin{align*}
    i\partial_t\tilde{u}_n+\Delta\tilde{u}_n=-\psi_n-|\tilde{u}_n|^{\frac{4}{3}}\tilde{u}_n-\mu A(|\tilde{u}_n|^2)|\tilde{u}_n-H
\end{align*}
with
\begin{align*}
    i\partial_t\tilde{R}_{\mathcal{P}_n  }+\Delta\tilde{R}_{\mathcal{P}_n  }+|\tilde{R}_{\mathcal{P}_n  }|^{\frac{4}{3}}\tilde{R}_{\mathcal{P}_n  }+\mu A(|\tilde{R}_{\mathcal{P}_n  }|^2)\tilde{R}_{\mathcal{P}_n}=\psi_n,~~
    H=H_1+H_2,
\end{align*}
where
\begin{align*}
H_1=&|\tilde{R}_{\mathcal{P}_n}-\tilde{u}_n|^{\frac{4}{3}}(\tilde{R}_{\mathcal{P}_n}+\tilde{u})-|\tilde{R}_{\mathcal{P}_n}|^{\frac{4}{3}}\tilde{R}_{\mathcal{P}_n}+\mu A(|\tilde{R}_{\mathcal{P}_n}+\tilde{u}_n|^2)(\tilde{R}_{\mathcal{P}_n}+\tilde{u}_n)-\mu A(|\tilde{R}_{\mathcal{P}_n}|^2)|\tilde{R}_{\mathcal{P}_n},\\
    H_2=&-|\tilde{u}_n|^{\frac{4}{3}}\tilde{u}_n-\mu A(|\tilde{u}_n|^2)\tilde{u}_n.
\end{align*}
Hence, from the standard Strichartz bounds and the smoothing effect, we have
\begin{align}\label{minimal:step2dou}
   \|\nabla^{\frac{3}{2}}\tilde{u}_n\|_{L^{\infty}([t,t_n])L^2(\mathbb{R}^3)}
   \lesssim&\|\nabla^{\frac{3}{2}}\psi_n\|_{L^{2}_{[t,t_n]}L^{6/5}(\mathbb{R}^3)}+\|\langle x\rangle \nabla H_1\|_{L^2_{[t,t_n]}L^2(\mathbb{R}^3)}\notag\\
   &+\|\nabla^{\frac{3}{2}}(|\tilde{u}_n|^{\frac{4}{3}}\tilde{u}_n)\|_{L^{3/2}_{[t,t_n]}L^{18/13}(\mathbb{R}^3)}+\mu\|\nabla^{\frac{3}{2}}(A(\tilde{u}_n^2)\tilde{u}_n)\|_{L^{8/5}_{[t,t_n]}L^{4/3}(\mathbb{R}^3)}.
\end{align}
The error term $\psi_n$ is estimated from  \eqref{equ:approximate:double} and lemma \ref{lemma:modestimate}, which yields a bound
\begin{align*}
    \|\nabla^{\frac{3}{2}}\psi_n\|_{L^{6/5}(\mathbb{R}^3)}\lesssim\|\nabla^2\psi_n\|_{L^2}^{\frac{1}{4}}\|\psi_n\|_{L^2}^{\frac{3}{4}}\lesssim\lambda_n^{\frac{3}{2}},
\end{align*}
where we used the Gagliardo-Nirenberg's inequality. Hence,
\begin{align}\label{minimal22dou}
    \|\nabla^{\frac{3}{2}}\psi_n\|_{L^{2}_{[t,t_n]}L^{6/5}(\mathbb{R}^3)}\lesssim\lambda_n^{2}.
\end{align}
For the term $H_1$, we  have
\begin{align*}
	|H_1|\lesssim |\tilde{R}_{\mathcal{P}_n}^{\frac{4}{3}}\tilde{u}_n|+|\tilde{u}_{n}^{\frac{4}{3}}\tilde{u}_n|+\mu\left[A\left(|\tilde{R}_{\mathcal{P}_n}|^2\right)\tilde{u}_n+A\left(|\tilde{R}_{\mathcal{P}_n}\tilde{u}_n|\right)|\tilde{R}_{\mathcal{P}_n}|+A\left(|\tilde{u}_n|^2\right)|\tilde{u}_n|\right]
	=I+II,
\end{align*}
where
\begin{align*}
	 I=|\tilde{R}_{\mathcal{P}_n}^{\frac{4}{3}}\tilde{u}_n|+\mu\left(A\left(|\tilde{R}_{\mathcal{P}_n}|^2\right)\tilde{u}_n+A\left(|\tilde{R}_{\mathcal{P}_n}\tilde{u}_n|\right)|\tilde{R}_{\mathcal{P}_n}|\right),~~
	II=|\tilde{u}_{n}^{\frac{4}{3}}\tilde{u}_n|+\mu A\left(|\tilde{u}_n|^2\right)|\tilde{u}_n|.
\end{align*}
We now estimate the local term that contain the linear terms of $\tilde{u}_n$,
\begin{align*}
    \|\langle x\rangle \nabla(|\tilde{R}_{\mathcal{P}_n}|^{\frac{4}{3}}\tilde{u}_n)\|_{H^1(\mathbb{R}^3)}
    \lesssim& \|\langle x\rangle (|\tilde{R}_{\mathcal{P}_n}|^{\frac{4}{3}}\tilde{u}_n)\|_{L^2}+\|\langle x\rangle |\tilde{R}_{\mathcal{P}_n}|^{\frac{1}{3}}\tilde{u}_n\nabla\tilde{R}_{\mathcal{P}_n}\|_{L^2}+ \|\langle x\rangle |\tilde{R}_{\mathcal{P}_n}|^{\frac{4}{3}}\nabla\tilde{u}_n\|_{L^2}\notag\\
    \lesssim&\frac{1}{\lambda_n}\|\tilde{u}_n\|_{L^2}+\frac{1}{\lambda_n^2}\|\tilde{u}_n\|_{L^2}+\frac{1}{\lambda_n}\|\nabla\tilde{u}_n\|_{L^2} \lesssim\lambda_n^2,
\end{align*}
where we used  $\tilde{u}_n(t_n)=0$, Lemma \ref{lemmabackwardsdou} and the decay estimate of $R_{\mathcal{P}_n}$.

Next, we estimate the non-local terms of I, by using the following estimate
\begin{align*}
	 A(|f|)=&\int_{|y|\leq\frac{|x|}{2}}\frac{|f(y)|}{|x-y|^2}dy+\int_{\frac{|x|}{2}<|y|<2|x|}\frac{|f(y)|}{|x-y|^2}dy+\int_{|y|\geq2|x|}\frac{|f(y)|}{|x-y|^2}dy\\
	 \lesssim&|x|^{-2}\|f\|_{L^1}+|x|^{-1}\int_{\frac{|x|}{2}<|y|<2|x|}\frac{|y|}{|x-y|^2}|f(y)|dy+|x|^{-2}\int_{|y|\geq2|x|}\frac{|y|^2}{|x-y|^2}|f(y)|dy\\
	 \lesssim&|x|^{-2}\|f\|_{L^1}+|x|^{-1}\left\|\frac{|y|}{|x-y|^2}\right\|_{L^2(|x|/2<|y|<2|x|)}\|f\|_{L^2}+|x|^{-2}\|f\|_{L^1}\\
	\lesssim&|x|^{-2}\|f\|_{L^1}+|x|^{-1}\left\||y|^{\frac{1}{2}}f\right\|_{L^2}.
\end{align*}
we can obtain 
\begin{align*}
	&\left\|\langle x\rangle \left[A\left(|\tilde{R}_{\mathcal{P}_n}|^2\right)\tilde{u}_n+A\left(|\tilde{R}_{\mathcal{P}_n}|\tilde{u}_n\right)|\tilde{R}_{\mathcal{P}_n}|\right]\right\|_{H^1}\\
	\lesssim&\left\|\langle x\rangle\left(\tilde{u}_n\nabla A\left(|\tilde{R}_{\mathcal{P}_n}|^2\right)\right)\right\|_{L^2}+\left\|\langle x\rangle  A\left(|\tilde{R}_{\mathcal{P}_n}|^2\right)\nabla\tilde{u}_n\right\|_{L^2}\\
	&+\left\|\langle x\rangle\left[|\tilde{R}_{\mathcal{P}_n}|\nabla A\left(|\tilde{R}_{\mathcal{P}_n}|\tilde{u}_n\right)+A\left(|\tilde{R}_{\mathcal{P}_n}|\tilde{u}_n\right)\nabla|\tilde{R}_{\mathcal{P}_n}|\right]\right\|_{L^2}\\
	&+\left\|\langle x\rangle\left(\tilde{u}_n A\left(|\tilde{R}_{\mathcal{P}_n}|^2\right)\right)\right\|_{L^2}+\left\|\langle x\rangle|\tilde{R}_{\mathcal{P}_n}| A\left(|\tilde{R}_{\mathcal{P}_n}|\tilde{u}_n\right)\right\|_{L^2}\\
	\lesssim&\|\langle x\rangle\nabla A(|\tilde{R}_{\mathcal{P}_n}|^2)\|_{L^6}\|\tilde{u}_n\|_{L^3}+\|\langle x\rangle A(\tilde{R}_{\mathcal{P}_n}^2)\|_{L^{\infty}}\|\nabla u_n\|_{L^2}+\frac{1}{\lambda_n^{1/2}}\|\nabla A(\tilde{R}_{\mathcal{P_n}}\tilde{u}_n)\|_{L^2}\\
	&+\frac{1}{\lambda_n^{3/2}}\|A(\tilde{R}_{\mathcal{P_n}}\tilde{u}_n)\|_{L^2}+\|\langle x\rangle A(|\tilde{R}_{\mathcal{P}_n}|^2)\|_{L^6}\|\tilde{u}_n\|_{L^3}+\frac{1}{\lambda_n^{1/2}}\|A(\tilde{R}_{\mathcal{P_n}}\tilde{u}_n)\|_{L^2}\\
	\lesssim&\frac{1}{\lambda_n^{1/2}}\|\nabla \tilde{R}_{\mathcal{P}}\|_{L^2}\|\tilde{u}_n\|_{L^3}+\left(\|\tilde{R}_{\mathcal{P}_n}\|_{L^2}^2+\||y|^{\frac{1}{2}}\tilde{R}_{\mathcal{P}_n}^2\|_{L^2}\right)\|\nabla \tilde{u}_n\|_{L^2}+\frac{1}{\lambda_n^{1/2}}\|\nabla(\tilde{R}_{\mathcal{P}_n}\tilde{u}_n)\|_{L^{6/5}}\\
	&+\frac{1}{\lambda_n^{3/2}}\|(\tilde{R}_{\mathcal{P}_n}\tilde{u}_n)\|_{L^{6/5}}+\frac{1}{\lambda_n^{1/2}}\|\nabla \tilde{R}_{\mathcal{P}}\|_{L^2}\|\tilde{u}_n\|_{L^3}+\frac{1}{\lambda_n^{1/2}}\|(\tilde{R}_{\mathcal{P}_n}\tilde{u}_n)\|_{L^{6/5}}\\
	 \lesssim&\frac{1}{\lambda_n^{\frac{3}{2}}}\|\tilde{u}_n\|_{L^2}+\left(1+\frac{1}{\lambda_n}\|\tilde{R}_{\mathcal{P}_n}\|_{L^3}^{\frac{3}{2}}\right)\|\nabla \tilde{u}_n\|_{L^2}+\frac{1}{\lambda_n^{1/2}}\left(\|\nabla \tilde{R}_{\mathcal{P}_n}\|_{L^2}\|\tilde{u}_n\|_{L^3}+\|\tilde{R}_{L^3}\|\nabla\tilde{u}\|_{L^2}\right)\\
	 &+\frac{1}{\lambda_n^{3/2}}(\|\tilde{R}_{\mathcal{P}_n}\|_{L^3}\|\tilde{u}_n\|_{L^2})+\frac{1}{\lambda_n^{1/2}}\|\nabla \tilde{R}_{\mathcal{P}}\|_{L^2}\|\tilde{u}_n\|_{L^3}+\frac{1}{\lambda_n^{1/2}}\|\tilde{R}_{\mathcal{P}_n}\|_{L^3}\|\tilde{u}_n\|_{L^2}\\
	\lesssim&\lambda_n^{\frac{1}{2}}.
\end{align*}
Thus, we have
\begin{align}\label{minimal23dou}
   \|\langle x\rangle \nabla I\|_{L^2_{[t,t_n]}L^2(\mathbb{R}^3)}\lesssim\lambda_n.
\end{align}
The local nonlinear term is estimated from Sobolev embedding and standard nonlinear estimates in Besov spaces,
\begin{align}\label{minimal24dou}
    \|\nabla^{\frac{3}{2}}(|\tilde{u}_n|^{\frac{4}{3}}\tilde{u}_n)\|_{L^{18/13}(\mathbb{R}^3)}\lesssim\|\tilde{u}_n^{\frac{4}{3}}\|_{L^{9/2}}\|\nabla^{\frac{3}{2}}\tilde{u}_n\|_{L^2}\lesssim\lambda_n^4\|\nabla^{\frac{3}{2}}\tilde{u}_n\|_{L^2}.
\end{align}
For the non-local nonlinear term, we have
\begin{align}\label{minimal25dou}
    \|\nabla^{\frac{3}{2}}(A(\tilde{u}_n^2)\tilde{u}_n)\|_{L^{4/3}(\mathbb{R}^3)}\lesssim&\|\nabla^{\frac{3}{2}}A(\tilde{u}_n^2)|\tilde{u}_n|\|_{L^{4/3}}+\|A(\tilde{u}_n^2)\nabla^{\frac{3}{2}}\tilde{u}_n\|_{L^{4/3}}\notag\\
    \lesssim&\|\nabla^{\frac{3}{2}}A(\tilde{u}_n^2)\|_{L^{12/7}}\|\tilde{u}_n\|_{L^6}+\|A(\tilde{u}_n^2)\|_{L^4}\|\nabla^{\frac{3}{2}}\tilde{u}_n\|_{L^2}\notag\\
    \lesssim&(\|\tilde{u}_n\|_{L^{12/5}}\|\tilde{u}_n\|_{L^6}+\|\tilde{u}_n\|_{L^{24/7}}^2)\|\nabla^{\frac{3}{2}}\tilde{u}_n\|_{L^2}\notag\\
    \lesssim&\lambda_n^3\|\nabla^{\frac{3}{2}}\tilde{u}_n\|_{L^2}.
\end{align}
Injecting \eqref{minimal22dou}, \eqref{minimal23dou}, \eqref{minimal24dou} and \eqref{minimal25dou} into \eqref{minimal:step2dou}, we deduce
\begin{align*}
    \|\nabla^{\frac{3}{2}}\tilde{u}_n\|_{L^{\infty}([t,t_n])L^2(\mathbb{R}^3)}\lesssim\lambda_n+\lambda_n^5\|\nabla^{\frac{3}{2}}\tilde{u}_n\|_{L^{\infty}([t,t_n])L^2(\mathbb{R}^3)},
\end{align*}
and \eqref{minimal5:2dou} holds.
\end{proof}

\appendix

\section{{Proof of the lemma  \ref{lemma1dou} and lemma \ref{lemma12dou}}}\label{sectionlemma}
The similar argument can be found in \cite{LMR2016RMI}. For the reader's convenience, we give a detail of adaption to our case. First, we prove the lemma \ref{lemma1dou}.

\begin{proof}
By contradiction, assume that there exists a blow up solution $u(t)$ of \eqref{equ:double} with $\mu<0$ and $\|u(t)\|_{L^2}=\|Q\|_{L^2}$. Let a sequence $t_n\rightarrow T^*\in(0,+\infty]$ with $\|\nabla u(t_n)\|_{L^2}\rightarrow+\infty$ and consider the renormalized sequence
\begin{align*}
    v_n(x)=a(t_n)^{\frac{3}{2}}u\big(a(t_n)^2t_n,a(t_n)x\big),\,\,a(t_n)=\frac{\|\nabla Q\|_{L^2}}{\|\nabla u(t_n)\|_{L^2}}.
\end{align*}
Then $a(t_n)\rightarrow0$ as $t\rightarrow T^*$.
By the conservation of mass $\|v_n\|_{L^2}=\|Q\|_{L^2}$ and conservation of energy and $\mu<0$,
\begin{align*}
    E_{\mu}(u_0)=E_{\mu}(u_n)\geq E_0(u_n)=\frac{E_0(v_n)}{a^2(t_n)}.
\end{align*}
This means
\begin{align*}
    {E_0(v_n)\leq a^2(t_n) E_{\mu}(u_0)\rightarrow 0\,\,\text{as}\,\,t_n\rightarrow T^*.}
\end{align*}
Therefore, the sequence $v_n$ is uniformly bounded in $H^1$ and it satisfie
\begin{align*}
    \|v_n\|_{L^2}=\|Q\|_{L^2},\,\,\|\nabla v_n\|_{L^2}=\|\nabla Q\|_{L^2},\,\,\limsup_{n\rightarrow+\infty}E_0(v_n)\leq0.
\end{align*}
From the standard concentration compactness argument, see \cite{MR2005Ann,W1982CMP}, up to a subsequence, for some $x_n\in\mathbb{R}^3$, $\gamma_n\in\mathbb{R}$,
\begin{align*}
    v_n(\cdot-x_n)e^{i\gamma_n}\rightarrow Q\,\,\text{in}\,\,H^1(\mathbb{R}^3)\,\,\text{as}\,\,n\rightarrow+\infty.
\end{align*}
This gives the lower bound
$$  \int A(v^2(t_n))|v(t_n)|^2 \geq \frac{1}{2} \int A(Q^2)|Q|^2 >0.$$
Therefore
\begin{align*}
    \int A(u^2(t_n))|u(t_n)|^2=\frac{ \int A(v^2(t_n))|v(t_n)|^2}{a^{2}(t_n)}\rightarrow+\infty\,\,\text{as}\,\,n\rightarrow+\infty,
\end{align*}
which contradicts the upper  bound of
$  \int A(u^2(t_n))|u(t_n)|^2$
following from
\begin{align*}
-\frac{\mu}{4}\int A(u^2(t_n))|u(t_n)|^2= E_\mu(u(t_n)) -E_0(u(t_n)) = E_\mu(u(0)) - \frac{E_0(v(t_n))}{a(t_n)^2} \lesssim 1.
\end{align*}
This concludes the proof.
\end{proof}

Next, we prove the lemma \ref{lemma12dou}.

\begin{proof}
Let $\delta>0$. We first recall the Virial identity (see \cite{Cazenave:book,L2020JDE,F2015EECT})
\begin{align*}
    \frac{d^2}{dt^2}\|xu(t)\|_{L^2}^2=8\left(\|\nabla u\|_{L^2}^2-\frac{3}{5}\|u\|_{L^{\frac{10}{3}}}^{\frac{10}{3}}\right)-2\mu\int A(u^2)|u|^2dx=16E(u).
\end{align*}
Define for $\alpha>0, \beta>0$ the rescaled function $Q_{\alpha,\beta}$ by $Q_{\alpha,\beta}(x)=\alpha^{\frac{3}{2}}Q(\alpha\beta x)$. We have
\begin{align*}
    \|Q_{\alpha,\beta}\|_{L^2}=\beta^{-\frac{3}{2}}\|Q\|_{L^2},~
    E_{\mu}(Q_{\alpha,\beta})=\alpha^2\beta^{-3}\left(\beta^{2}\frac{1}{2}\|\nabla Q\|_{L^2}^2-\frac{3}{10}\|Q\|_{L^{\frac{10}{3}}}^{\frac{10}{3}}-\frac{\mu}{4}\beta^{-1}\int A(Q^2)|Q|^2\right).
\end{align*}
Recall that the critical energy vanishes at $Q$, that is,
\begin{align*}
    \frac{1}{2}\|\nabla Q\|_{L^2}^2-\frac{3}{10}\|Q\|_{L^{\frac{10}{3}}}^{\frac{10}{3}}=0.
\end{align*}
Take $1>\beta>\frac{1}{2}$ such that $\|Q_{\alpha,\beta}\|_{L^2}=\|Q\|_{L^2}+\delta$. Then for any $\alpha>0$  and $\mu<0$ very close to $0$, we have
\begin{align*}
    E(Q_{\alpha,\beta})<0.
\end{align*}
Choosing $u_0=Q_{\alpha,\beta}$. By the conservation of energy and the Virial identity, we have
\begin{align*}
    \frac{d^2}{dt^2}\|xu(t)\|_{L^2}^2=16 E(Q_{a,b})<0,
\end{align*}
which implies blowup in finite time of $u$ for positive and negative times.
\end{proof}

\section{Non-degeneracy in the Nonradial sector}\label{appendixnondegeneracy}
Since $Q$ is a radial function, the operator $L_+$ commutes with rotations on $\mathbb{R}^3$. Using the decomposition in terms of spherical harmonics
\begin{align}\notag
    L^2(\mathbb{R}^3)=\bigoplus_{l\geq0}\mathcal{H}_l,
\end{align}
we find that $L_+$ acts invariantly on each subspace
\begin{align}\notag
    \mathcal{H}_l=L^2(\mathbb{R}_+,r^2dr)\bigotimes\mathcal{Y}_l.
\end{align}
Here $\mathcal{Y}_l=span\{Y_{lm}\}_{m=-l}^{+l}$ denotes the space of the spherical harmonics pf degree $l$ in space dimension $4$. Recall also that 
\begin{align}\notag
    -\Delta_{\mathbb{S}^2}Y_{lm}=l(l+2)Y_{lm}\,\,\text{and}\,\,\langle Y_{lm_1},Y_{lm_2}\rangle_{L^2(\mathbb{S}^2)}=\begin{cases}
    1,\,\,&\text{if}\,\,m_1=m_2,\\
    0,&\text{if}\,\,m_1\neq m_2,
    \end{cases}
\end{align}
where $\Delta_{\mathbb{S}^3}$ is the Laplacian on $\mathbb{S}^2$. Therefore, we decompose any $\xi\in L^2(\mathbb{R}^3)$ using spherical harmonics according to 
\begin{align}\notag
    \xi(x)=\sum_{l=0}^\infty\sum_{m=-l}^lf_{lm}(r)Y_{lm}.
\end{align}

Let us now find an explicit formula for the action of $L_+$ on each $\mathcal{H}_l$. To this end, we recall the well-known the fact that 
\begin{align}\notag
    -\Delta=-\partial_r^2-\frac{2}{r}\partial_r+\frac{l(l+1)}{r^2}.
\end{align}
By using the multipole expansion, see \cite[Example 3.5.12]{Simon:book}, we have
the following expansion of the nonlocal term 
\begin{align}\notag
    (|x|^{-2}*\psi)(x)=\sum_{l=0}^\infty\sum_{m=-l}^l\frac{1}{2m+2}\left(\int_0^\infty\frac{(r\wedge\rho)^l}{(r\vee \rho)^{l+2}}f_{lm}(\rho)\rho^2d\rho\right)Y_{lm}(\theta),
\end{align}
where $r\wedge\rho=\min\{r,\rho\}$, $r\vee\rho=\max\{r,\rho\}$, $r=|x|$, $\theta=\frac{x}{|x|}$ and $\psi(x)=\psi(r,\theta)=\sum_{l=0}^\infty\sum_{m=-l}^lf_{lm}Y_{lm}(\theta)$ is an expansion of $\psi$ by spherical harmonics. An elementary calculation leads to the following equivalence: We have that $L_+\xi=0$ if and only if 
\begin{align}\notag
    L_{+,l}f_{lm}=0,\,\,\text{for}\,\,l=0,1,2,...\text{and}\,\,m=-l,\ldots,+l.
\end{align}
Here the operator $L_{+,l}$ acting on $L^2(\mathbb{R}_+,r^2dr)$ is given by 
\begin{align}\notag
    (L_{+,l}f)(r)=-f^{\prime\prime}(r)-\frac{2}{r}f^\prime(r)+\frac{l(l+1)}{r^2}f(r)+V(r)f(r)+(W_{l}f)(r),
\end{align}
with the local potential 
$V(r)=-\frac{7}{3}Q^{\frac{4}{3}}(r)-A(Q^2)(r)$,
and the nonlocal linear operator
\begin{align}\notag
    (W_{l}f)(r)=-\frac{2}{2l+2}Q(r)\int_0^\infty\frac{(r\wedge\rho)^l}{(r\vee \rho)^{l+2}}Q(\rho)\left( \sum_{m=-\ell}^\ell f_{lm}(\rho)\right)\rho^2d\rho,
\end{align}
where $r\wedge\rho=\min\{r,\rho\}$, $r\vee\rho=\max\{r,\rho\}$,

Note that lemma \ref{lemma:raidalnondegeneracy} above says that $\ker L_{+,0}=\{0\}$. We now derive the following result.
\begin{lemma}\label{lemma:nondegeneracy}
We have $\ker L_{+,1}=span\{\partial_rQ\}$ and $\ker L_{+,l}=\{0\}$ for $l\geq2$.
\end{lemma}
\begin{proof}
By the similar argument as \cite[Lemma 7]{L2009APDE}, for each $l\geq1$ the operator $L_{+,l}$ is essentially self-adjoint on $C_0^\infty(\mathbb{R}_+)\subset L^2(\mathbb{R}_+,r^3dr)$ and bounded below. Moreover, each $L_{+,l}$ has the Perron-Frobenius property. That is, if $e_{0,l}$ denotes the lowest eigenvalue of $L_{+,l}$, then  $e_{0,l}$ is simple and the corresponding eigenfunction $\psi_{0,l}(r)>0$ is strictly  positive.

By differentiating the nonlinear equation \eqref{equ:double} satisfied by $Q$, we readily obtain that $L_+\partial_{x_i}Q=0$ for $i=1,2,3$. Since $\partial_{x_i}Q(r)=Q^\prime(r)(\frac{x_i}{r})\in \mathcal{H}_1$, this show that 
\begin{align}\notag
    L_{+,1}Q^\prime=0.
\end{align}
Furthermore, by monotonicity of $Q(r)$, we have that $Q^\prime(r)\leq0$. Since $L_{+,1}$ is self-adjoint and $Q^\prime$ is an eigenfunction that does not change its sign, the above result shows that $Q^\prime(r)=-\psi_{0,1}$ holds, where $\psi_{0,1}>0$ is strictly positive ground state of $L_{+,1}$ with $e_{+,1}=0$ being its corresponding eigenvalue. Since $\mathcal{H}_1$ has dimension 3 any $\xi\in\mathcal{H}_1$  must be some linear combination of $\{\partial_{x_i}Q\}_{i=1}^3$. Hence, we have $\ker L_{+,1}=span\{\partial_rQ\}.$ 

To complete the proof of this lemma, we now claim that 
\begin{align}\label{claim:nonradial}
    L_{+,l}>0,\,\,\text{for}\,\,l\geq2.
\end{align}
Indeed, by the similar argument as \cite[Proposition 4]{L2009APDE}, we can obtain \eqref{claim:nonradial}.
This implies that if $L_{+}\xi=0$ with $\xi\in\mathcal{H}_{l}$ for some $l\geq2$, then  $\xi\equiv0$, which completes the proof of lemma \ref{lemma:nondegeneracy}.
\end{proof}

{\bf Proof of Non-degeneracy Lemma \ref{nondegeneracy}:}
Let $\xi\in\mathbb{R}^3$ satisfy $L_+\xi=0$. By lemma \ref{lemma:raidalnondegeneracy} and \ref{lemma:nondegeneracy}, we conclude that $\xi\in \mathcal{H}_{l=1}$ and that $\xi$ must be a linear combination of $\partial_{x_1}Q,\partial_{x_2}Q,\partial_{x_3}Q$. Now we completes the proof of Lemma \ref{nondegeneracy}.

{\bf Acknowledgments}

VG was partially supported by GNAMPA - Gruppo Nazionale per l'Analisi Matematica, la Probabilita e le loro Applicazioni, by Institute of Mathematics and Informatics, Bulgarian Academy of Sciences, by Top Global University Project, Waseda University and by the project PRIN  2020XB3EFL funded by the Italian Ministry of Universities and Research. YL was supported by China Postdoctoral Science Foundation (No. 2021M701365).


\renewcommand{\proofname}{\bf Proof.}




\vspace*{.5cm}


\begin{thebibliography}{}
\bibitem{BCD2011CPDE}  V. Banica, R. Carles, T. Duyckaerts, Minimal blow-up solutions to the mass-critical inhomogeneous NLS equation. Comm. Partial Differential Equations {\bf 36} (3) (2011), 487 -- 531.


\bibitem{Schlein:book} N. Benedikter, M. Porta, B. Schlein, Effective evolution equations from quantum dynamics, Springer Briefs in Mathematical Physics, Vol. 7 Springer, Cham, 2016.


\bibitem{BL1983ARMA}H. Berestychi, P. L. Lions, Nonlinear scalar field equations. I. Existence of a ground state. Arch. Rational Mech. Anal. {\bf  82} (1983), 313--345. 
\bibitem{Cazenave:book} T. Cazenave, Semilinear Schr\"{o}dinger equations. Courant Lecture Notes in Mathematics, 10. New York University, Courant Institute of Mathematical Sciences, New York; American Mathematical Society, Providence, R.I. ( 2003)
\bibitem{CGN2007SIAM} S.M. Chang, S. Gustafson, K. Nakanishi,  T.P. Tsai, Spectra of linearized operators for NLS solitary waves. SIAM J. Math. Anal. {\bf 39} (2007), 1070 -- 1111. 
\bibitem{F2015EECT} B.H. Feng,  X, Yuan, On the Cauchy problem for the Schr\"{o}dinger-Hartree equation. Evol. Equ. Control Theory {\bf 4 } (4) (2015), 431--445. 

\bibitem{G2020}A. Giusti, MOND-like Fractional Laplacian Theory, Phys. Rev. D {\bf 101} (12) (2020), 124029.

\bibitem{I2013TMNA} I. Ianni, Sign-changing radial solutions for the Schr\"odinger–Poisson–Slater problem, Topol. Methods Nonlinear Anal. 41 (2013), 365--385.

\bibitem{GL2021CPDE} V. Georgiev, Y. Li. Nondispersive solutions to the mass critical half-wave equation in two dimensions. Communications in Partial Differential Equations, (2021), doi.org/10.1080/03605302.2021.1950763
\bibitem{GL2021JFA} V. Georgiev, Y. Li,  Blowup dynamics for mass critical half-wave equation in 3D. J. Funct. Anal. 281 (2021),  109132.
\bibitem{GS2018PD} V. Georgiev, A. Stefanov, On the classification of the spectrally stable standing waves of the Hartree equation. Phys. D {\bf 370} (2018), 29--39. 

\bibitem{GTV2019NA} V. Georgiev, M. Tarulli,  G. Venkov, Existence and uniqueness of ground states for p-Choquard model.
Nonlinear Anal. {\bf 179} (2019), 131--145.

\bibitem{GPV2012Poincare} V. Georgiev, F. Prinari, N. Visciglia, On the radiality of constrained minimizers to the Schr\"odinger–Poisson–Slater energy, Ann. Inst. H. Poincar\'e Anal. Non Lin\'eaire 29 (2012), 369–-376.

\bibitem{GGV2020} A. Giusti, R. Garrappa, G. Vachon, On the Kuzmin model in fractional Newtonian gravity,
Eur. Phys. J. Plus {\bf 135} (798) (2020). 

\bibitem{KLR2013ARMA} J. Krieger, E. Lenzmann,  P. Rapha\"{e}l, Nondispersive solutions to the $L^2$-critical half-wave equation. Arch. Rational Mech. Anal. {\bf 209} (2013), 61--129. 
\bibitem{KMR2009CPAM}J. Krieger, Y. Martel,  P. Rapha\"{e}l,  Two-soliton solutions to the three dimensional gravitational Hartree equation. Comm. Pure Appl. Math. {\bf 62} (11) (2009), 1501--1550. 
\bibitem{KLR2009poincare} J. Krieger, E. Lenzmann,  P. Rapha\"{e}l,  On stability of pseudo-conformal blowup for $L^2$-critical Hartree NLS. Ann. Henri Poincar\'{e} {\bf 10} (2009),   1159--1205.
\bibitem{K1989ARMA} M.K. Kwong, Uniqueness of positive solutions of $\Delta u-u+u^p=0$ in $\mathbb{R}^N$. Arch. Rational Mech. Anal. {\bf 105} (1989), 243--266. 
\bibitem{Lan2021IMRN} Y.Lan,
  Blow-up dynamics for $ L^{2}$-critical fractional {S}chr\"{o}dinger equations.
  Int. Math. Res. Notices,(2021) Doi:10.1093/imrn/rnab086
\bibitem{LMR2016RMI}S. Le Coz,  Y. Martel, P. Rapha\"{e}l, Minimal mass blow up solutions for a double power nonlinear Schr\"{o}dinger equation. Rev. Mat. Iberoam. {\bf  32} (2016), 795--833. 
\bibitem{LZW2019ZAMP} Y. Li, D. Zhao, Q. Wang, Concentration behavior of nonlinear Hartree-type equation with almost mass critical exponent. Z. Angew. Math. Phys.  {\bf 70} Art. 128 (2019).

\bibitem{L2020JDE}X. Li, Global existence and blowup for Choquard equations with an inverse-square potential. J. Differential Equations {\bf 268} (2020), 4276--4319. 
\bibitem{L1976SAM}  E.H. Lieb,  Existence and uniqueness of the minimizing solution of Choquard  nonlinear equation. Stud. Appl. Math.
{\bf 57} (1976/77), 93--105. 
\bibitem{L2001:book}E.H. Lieb, M. Loss, Analysis, Second ed., Graduate Studies in Math. 14, Amer. Math. Soc., Providence,
RI, 2001.
\bibitem{MZ2010ARMA} L. Ma, L. Zhao, Classification of positive solitary solutions of the nonlinear Choquard equation. Arch. Ration. Mech. Anal. {\bf 195} (2010), 455--467.  
\bibitem{L2009APDE}  E. Lenzmann,  Uniqueness of ground states for pseudorelativistic Hartree equations. Anal. PDE {\bf 2} (2009), 1--27. 

\bibitem{MP2017MA} Y. Martel, D. Pilod,  Construction of a minimal mass blow up solution of the modified Benjamin-Ono equation. Math. Ann. 369 (2017), no. 1-2, 153–-245.
\bibitem{M1990CMP} F. Merle,  Construction of solutions with exactly $k$ blow-up points for the Schr\"{o}dinger equation with critical nonlinearity. Comm. Math. Phys. {\bf 129} (2) (1990), 223--240.
\bibitem{M1993Duke} F. Merle, Determination of minimal blow-up solutions with minimal mass for nonlinear Schr\"{o}dinger equations with critical power. Duke Math. J. {\bf  69} (1993), 427--454. 

\bibitem{MRGFA2003} F. Merle, P. Rapha\"{e}l, Sharp upper bound on the blow-up rate for the critical nonlinear Schr\"{o}dinger equation, Geom. Funct. Anal. {\bf 13 } (3) (2003), 591--642.
\bibitem{MR2004Invent} F. Merle, P. Rapha\"{e}l,  On universality of blow-up profile for $L^2$ critical nonlinear Schr\"{o}dinger equation, Invent. Math. {\bf 156} (3) (2004), 565--672.
\bibitem{MR2005CMP} F. Merle, P. Rapha\"{e}l, Profiles and quantization of the blow up mass for critical nonlinear Schr\"{o}dinger equation, Comm. Math. Phys. {\bf 253} (3) (2005),  675--704. 


\bibitem{MR2005Ann} F. Merle, P. Rapha\"{e}l,  The blowup dynamic and upper bound on the blow-up rate for critical nonlinear Schr\"{o}dinger equation. Ann. of Math. (2) {\bf 161} (1) (2005), 157--222. 
\bibitem{MR2006JAMS} F. Merle, P. Rapha\"{e}l, On a sharp lower bound on the blow-up rate for the $L^2$ critical nonlinear Schr\"{o}dinger equation. J. Amer. Math. Soc. {\bf 19} (2006), 37--90. 

\bibitem{MRS2013Amer} F. Merle, P. Rapha\"{e}l, J. Szeftel, The instability of Bourgain-Wang solutions for the $L^2$ critical NLS. Amer. J. Math. {\bf 135}  (2013), 967--1017.
\bibitem{MRS2014Duke} F. Merle, P. Rapha\"{e}l,  J. Szeftel,  On collapsing ring blow-up solutions to the mass supercritical nonlinear Schr\"{o}dinger equation. Duke Math. J. {\bf 163} (2) (2014), 369--431.
\bibitem{MV2013JFA} V. Moroz, J. Van Schaftingen, Groundstates of nonlinear Choquard equations: Existence, qualitative properties and decay asymptotics, J.  Funct. Anal. {\bf 265}  (2013),  153--184. 


\bibitem{RS2011JAMS} P. Rapha\"{e}l,  J. Szeftel, Existence and uniqueness of minimal blow-up solutions to an inhomogeneous mass critical NLS. J. Amer. Math. Soc. {\bf 24} (2) (2011), 471--546.

\bibitem{R2010ARMA} D. Ruiz, On the  Schr\"odinger–Poisson–Slater system: behavior of minimizers, radial and nonradial cases, Arch. Ration. Mech. Anal. {\bf 198} (2010) 349--368.

\bibitem{SS2004JSP} O.  S\'anchez, J. Soler, Long-time dynamics of the Schro\"odinger–Poisson–Slater system, J. Statist. Phys. {\bf 114} (2004), 179 -- 204. 

\bibitem{Simon:book} B. Simon, Harmonic Analysis. American Mathematical Society, Providence, 2015.
\bibitem{S1977CMP} W.A. Strauss, Existence of solitary waves in higher dimensions. Commun. Math. Phys. {\bf 55} (1977), 149--162.

\bibitem{V2020FP} G.U. Varieschi, Newtonian fractional-dimension gravity and MOND. Found. Phys. {\bf 50} (2020), 1608 –-1644.
\bibitem{V2021EP} G.U. Varieschi, Newtonian Fractional Gravity and Disk Galaxies,
Eur. Phys. J. Plus {\bf 136} (183) (2021).

\bibitem{W1982CMP} M.I.  Weinstein, Nonlinear Schr\"{o}dinger equations and sharp interpolation estimates, Commun. Math. Phys. {\bf 87} (1983), 567--576.

\bibitem{W1985SIAM}  M.I. Weinstein, Modulational stability of ground states of nonlinear  Schr\"{o}dinger equations. SIAM J. Math. Anal.{\bf 16} (3) (1985), 472-491. 

\bibitem{X2016CVPDE} C.L. Xiang, Uniqueness and nondegeneracy of ground states for Choquard equations in three dimensions. Calc. Var. Partial Differential Equations {\bf 55} ,  Art. 134 (2016).
\end{thebibliography}

\bigskip

\begin{flushleft}
Vladimir Georgiev,\\
Dipartimento di Matematica, Universit\`{a} di Pisa, Largo B. Pontecorvo 5, 56127 Pisa, Italy\\
 Faculty of Science and Engineering, Waseda University, 3-4-1, Okubo, Shinjuku-ku, Tokyo 169-8555, Japan\\
 IMICBAS, Acad. Georgi Bonchev Str., Block 8, 1113 Sofia, Bulgaria\\
E-mail: georgiev@dm.unipi.it
\end{flushleft}

\begin{flushleft}
Yuan Li,\\
School of Mathematics and Statistics, Central China Normal University, Wuhan, PR China
\quad
E-mail: yli2021@mail.ccnu.edu.cn
\end{flushleft}

\bigskip

\medskip

\end{document}